\documentclass{amsart}
\usepackage{amssymb}
\usepackage{latexsym}
\usepackage[margin=1.4in]{geometry}
\usepackage{xcolor}
\usepackage{blkarray}  
\usepackage[utf8]{inputenc}
\usepackage{enumitem}

\date{\today}

\newcommand{\Z}{{\mathbb Z}}
\newcommand{\R}{{\mathbb R}}

\newcommand{\C}{{\mathbb C}}
\newcommand{\D}{{\mathbb D}}

\newcommand{\N}{{\mathbb N}}
\newcommand{\M}{{\mathbb M}}
\newcommand{\PP}{{\mathbb P}}

\newcommand{\CD}{{\mathcal D}}
\newcommand{\CE}{{\mathcal E}}
\newcommand{\CF}{{\mathcal F}}

\newcommand{\CS}{{\mathcal S}}

\newcommand{\CK}{{\mathcal K}}

\def\t{{\theta}}

\def\b{{\beta}}
\def\a{{\alpha}}
\def\e{{\varepsilon}}
\def\beq{\begin{equation}}
	\def\eeq{\end{equation}}

\newcommand{\SL}{{\mathrm{SL}}}

\let\strokel\l
\renewcommand\l{\lambda}

\newtheorem{theorem}{Theorem}
\newtheorem{remark}{Remark}
\newtheorem{lemma}{Lemma}

\newtheorem{defi}{Definition}

\newtheorem{corollary}{Corollary}
\newtheorem{prop}{Proposition}
\begin{document}
	
	\title[Domination and Jacobi Operators ]{On the Correspondence Between Domination and the Spectrum of Jacobi Operators}
	
		\author[K.\ Alkorn]{Kateryna Alkorn}
		
		\address{Department of Mathematics, University of California--Riverside, CA~92521, USA}
		
	\email{kateryna@math.ucr.edu}
	
	\author[Z.\ Zhang]{Zhenghe Zhang}
	
	\address{Department of Mathematics, University of California--Riverside, CA~92521, USA}
	
	\email{zhenghe.zhang@ucr.edu}

	\thanks{The authors were in part supported by NSF grant DMS-1764154.}
	
	\begin{abstract} In this paper, we first develop a notion of dominated splitting for $\M(2,\C)$-sequences and show it is a stable property under $\|\cdot \|_\infty$-perturbation. Then we show an energy parameter belongs to the spectrum of a Jacobi operator, possibly singular, if and only if the associated Jacobi cocycle does not admit dominated splitting. This generalizes the results obtained by the second author \cite{zhang2} in the scenario of Schr\"odinger operators. Finally, we consider dynamically defined Jacobi operators whose base dynamics is only assumed to be topologically transitive. We show an energy parameter belongs to the spectrum of the operator defined by the base point with a dense orbit if and only if the dynamically defined Jacobi cocycle does not admit dominated splitting. This includes the original Johnson's theorem obtained by R. Johnson \cite{johnson} for Sch\"odinger operators and the main theorem obtained by C. Marx \cite{marx} for Jacobi operators as special cases.
	\end{abstract}

	\maketitle
	\tableofcontents
	
	\section{Introduction and Statement of Main Results}
	
	In this paper we study the relationship between the spectrum of the Jacobi operator and the dynamics of its cocycle. We set $\ell^p(\Z)=\ell^p(\Z,\C)$ for $p\in\Z_+$ or $p=\infty$ and let $\ell^p(\Z,\R)\subset\ell^p(\Z)$ be real-valued sequences. The Jacobi operator $J_{a,b}:\ell^2(\Z)\rightarrow\ell^2(\Z)$ is given by
	\beq\label{eq:operators2}
	(J_{a,b} \psi)(n)=\overline{a_{n-1}}\psi(n-1)+a_n\psi(n+1)+b_n\psi(n),
	\eeq
	where $\psi=(\psi_n)_{n\in\Z}\in\ell^2(\Z)$, $a=(a_n)_{n\in\Z}\in\ell^{\infty}(\Z)$, and $b=(b_n)_{n\in\Z}\in\ell^{\infty}(\Z,\R)$. We say the Jacobi operator is singular if $a_n=0$ for some $n\in\Z$.
	
	Jacobi operators are bounded self-adjoint operators. Hence, the spectrum $\sigma(J_{a,b})$ of the operator $J_{a,b}$ which is defined as
	$$
	\sigma(J_{a,b}):=\{E \in \C : J_{a,b}-E \textit{ does not have bounded inverse} \}
	$$
	is a nonempty compact subset of $\R$. The resolvent, $\rho(J_{a,b})$, of the operator $J_{a,b}$ is defined as: 
	$$
	\rho(J_{a,b}) = \C- \sigma(J_{a,b}) 
	$$
	
	Another type of Jacobi operators we will study are dynamically defined operators which are given as follows. Let $\Omega$ be a compact metric space and $T$ be a homeomorphism on $\Omega$. Let $a\in C(\Omega,\C)$ and $b\in C(\Omega,\R)$, i.e. $a$ and $b$ are complex and real valued continuous functions on $\Omega$, respectively. Then for each $\omega\in\Omega$, we may define a Jacobi operator as 
	\beq\label{eq:dynamical_jacobi}
	(J_\omega\psi)(n)=\overline{a(T^{n-1}\omega)}\psi(n-1)+a(T^n\omega)\psi(n+1)+b(T^n\omega)\psi(n).
	\eeq
	In this case, we let $\sigma(J_\omega)$ and $\rho(J_\omega)$ denote its spectrum and resolvent set, respectively.
	
Jacobi operators are natural extension of one-dimensional discrete Schr\"odinger operators, the definition of which is to set $a_n=1$ for all $n\in\Z$ in \eqref{eq:operators2}. Just like Schr\"odinger operators, Jacobi operators arise naturally in various ways in mathematics and physics. For instance, they play a fundamental role in the study of completely integrable nonlinear lattices, in particular the Toda lattice and its modified counterpart, the Kac-van Moerbeke lattice. In the study of inverse spectral theory, one has to go to Jacobi operators even if one starts with Schr\"odinger operators. We refer the readers to \cite{teschl} for more information regarding Jacobi operators.

In many ways,  the study of spectral analysis of Schr\"odinger or Jacobi operators may be reduced to the study of  dynamics of the associated cocycles. One of the basic relations between spectral analysis of operators and dynamics of cocycles is the Johnson's type of theorems where one can identify the spectrum by these energies whose cocycles admit certain dynamics structure. For Schr\"odinger operators, this dynamics structure is the so-called uniform hyperbolicity. In the scenario of Jacobi operators, it turns out that the dynamics structure is the so-called dominated splitting which generalizes uniform hyperbolicity. These type of correspondences in particular play a key role in the analysis of the so-called Cantor spectrum phenomenon. For works related to these topics, we refer the readers to \cite{aviladamanikbochi1, aviladamanikbochi2, wangzhang} in the Schr\"odinger case and to \cite{fillmanongzhang, avilajitomirskayamarx, jitomirskayamarx, marx} in the Jacobi case and the references therein. Cantor spectrum phenomenon, on the other hand, has a deep physcis background. It is closely related to the so-called quantum hall effect, see e.g. \cite{hofstadter, thouless2}. The present paper concerns Johnson type of theorem for Jacobi operators. In the next two subsections, we first review the Johnson's theorem for Schr\"odinger operators. Then we discuss the existing Johnson's type of results for Jacobi operators due to C. Marx \cite{marx}. Finally, we state  the main results of this paper and discuss the strategy of their proofs.

\subsection{Review of Johnson's theorem for Schr\"odinger operators}\label{ss:johnson}
Let $b\in C(\Omega,\R)$. Consider the Schr\"odinger operators $H_\omega$ with potentials given by $b_n=b(T^n\omega)$, $n\in\Z$. Again, to define such operators, we just need to set $a(\omega)=1$ for all $\omega\in\Omega$ in \eqref{eq:dynamical_jacobi}. Let $A^E:\Omega\to\SL(2,\R)$ be the Schr\"odinger cocycle map at the energy $E$. For $\omega\in\Omega$, we define its orbit under $T$ to be $\mathrm{Orb}(\omega):=\{T^n\omega: n\in\Z\}$. Then the following theorem goes back to \cite[Theorem 3.1]{johnson}:

\begin{theorem}\label{t:johnson}
	Let $(T, \Omega)$ be topologically transitive. In other words, there is a $\omega_0\in\Omega$ such that $\overline{\mathrm{Orb}(\omega_0)}=\Omega$. Then $(T,A^E)$ is uniformly hyperbolic if and only if $E\in\rho(H_{\omega_0})$.
	\end{theorem} 
	
	We wish to point out that Johnson's theorem does not necessarily need to involve a base dynamics $(T,\Omega)$. Correspondingly, the notion of uniform hyperbolicity may be defined for a sequence of $\SL(2,\R)$-matrices.  Let $H_b$ be the Schr\"odinger operator with the potential $b\in\ell^\infty(\Z)$ and let $A^E:\Z\to\SL(2,\R)$ be the cocycle map defined on $\Z$. Then the following theorem is from \cite[Theorem 3]{zhang2}:
	
	\begin{theorem}\label{t:johnson_sequence}
    $A^E:\Z\to\SL(2,\R)$ is uniformly hyperbolic if and only if $E\in\rho(H_{b})$.
	\end{theorem} 
	
	Assume $(T,A^E)$ is a uniformly hyperbolic cocycle defined over a base dynamics $(T,\Omega)$. Then for each $\omega\in\Omega$ $A^E:\Z\to\SL(2,\R)$, where $A^E(n)=A^E(T^n\omega)$, is a uniformly hyperbolic $\SL(2,\R)$-sequence. Thus the following corollary is a direct consequence of Theorem~\ref{t:johnson_sequence}:
	
	\begin{corollary}\label{c:johnson_no_transitivity}
		If $(T,A^E)$ is uniformly hyperbolic, then $E\in \bigcap_{\omega\in\Omega}\rho(H_\omega)$.
		\end{corollary}
	Note that Corollary~\ref{c:johnson_no_transitivity} does not have any restrictions on the base dynamics. It is also clear that  Theorem~\ref{t:johnson_sequence} implies: \\
	
	\textit{If $E\in\bigcap_{\omega\in\Omega}\rho(H_\omega)$, then $A^{E}(T^{(\cdot)}\omega):\Z\to\SL(2,\R)$ is uniformly hyperbolic for all $\omega\in\Omega$.}\\
	
	In general, for a continuous cocycle map $A:\Omega\to\SL(2,\R)$, one may not be able to lift the uniform hyperbolicity from sequences (i.e. uniform hyperbolicity of $A(T^{(\cdot)}\omega):\Z\to\SL(2,\R)$ for all $\omega\in\Omega$) to uniform hyperbolicity of  cocycle $(T,A)$. This is because one may lose the uniformity of certain constants. However, when the base dynamics is topologically transitive, one can indeed pass the uniform hyperbolic from sequences to the cocycle. In fact, if $\overline{\mathrm{Orb}(\omega_0)}=\Omega$, then the uniform hyperbolicity of $A(T^{(\cdot)}\omega):\Z \to\SL(2,\R)$ implies the uniform hyperbolicity of $(T,A)$. This is due the fact that uniform hyperbolicity is equivalent to a uniform exponential growth condition which can easily be passed uniformly from a dense orbit to the whole space.
	
	Alternatively, one may use the following relatively simple fact, see e.g. \cite[Theorem 6]{zhang2}, to show that if $E\in\rho(H_{\omega_0})$, then $(T,A^E)$ is uniformly hyperbolic.
	\begin{prop}\label{p:dense_orbit_spectrum}
		If $\overline{\mathrm{Orb}(\omega_0)}=\Omega$, then $\sigma(H_\omega)\subset\sigma(H_{\omega_0})$ for all $\omega\in\Omega$.
		\end{prop}
	In particular, if $E\in \rho(H_{\omega_0})$, then $d(E,\sigma(H_\omega))\ge d(E,\sigma(H_{\omega_0}))>0$ for all $\omega\in\Omega$ where $d(E,\sigma(H_\omega))$ denotes the distance between $E$ and $\sigma(H_\omega)$. Then using the strategy of \cite[Section 3.3.2]{zhang2}, one can show that the constants appear in the definition of uniform hyperbolicity of $A(T^{(\cdot)}\omega):\Z \to\SL(2,\R)$
	depend only on $\|(H_\omega-E)^{-1}\|^{-1}=d(E,\sigma(H_\omega))$. In particular, those constants can be made uniform for all $\omega\in\Omega$.
	
	\subsection{The case with Jacobi operators and strategy of proofs}
	
	For Jacobi operators, the notion of uniform hyperbolicity is not sufficient as the associated cocycle map takes values in $\M(2,\C)$, which denotes the set of all $2\times 2$ matrices. Not only they might be not in $\SL(2,\R)$, they might even be singular, i.e. have zero determinant. Motivated by the  study of the extended Harper model \cite{jitomirskayamarx, avilajitomirskayamarx},  C. Marx \cite{marx} first noted that in such scenario the right choice should be cocycles admitting dominated splitting, which is a natural generalization of uniformly hyperbolic cocycles to those allowing $\M(2,\C)$-values. Domination, on the other hand, is an intensively studied notion in dynamical systems. It was introduced by and plays a key role in the works of Ma\~ n\'e \cite{mane}  and Liao \cite{liao} on Smale' s stability conjecture. The term \textit{dominated splitting} was introduced by Ma\~n\'e in \cite{mane}. We refer the readrs to \cite{bochiviana, bonatti, pujals} and the referenes therein regarding recent works in smooth dynamical systems involving domination. For recent works related to domination for cocyles defined over base dynamics, we refer the readers to \cite{avilajitomirskayasadel, bochigourmelon, blumenthalmorris} and the references therein.

	Regarding the correspondence between domination and spectrum of Jacobi operators, Marx \cite{marx} showed the following version of Johnson's theorem. We again let $T:\Omega\to\Omega$ be homeomorphism on a compact metric space $\Omega$. We say $T$ is minimal if $\overline{\mathrm{Orb}(\omega)}=\Omega$ for all $\omega\in\Omega$. We say $T$ is uniquely ergodic if $T$ has only one ergodic probablity measure. We say $T $ is strictly ergodic if $T$ is uniquely ergodic and minimal. By Proposition~\ref{p:dense_orbit_spectrum_J}, $\rho(J_\omega)$ is independent of $\omega\in\Omega$ if $T$ is minimal. Let $B^E:\Omega\to\M(2,\C)$ be the Jacobi cocycle map that is associated with the energy $E\in\C$. 
	 
	\begin{theorem}\label{t:marx}
		Let $T:\Omega\to\Omega$ be strictly ergodic. Assume $B^E\in C^0(\Omega,\M(2,\C))$ and $\int_\Omega\big|\log|\det(B^E(\omega))|\big|d\mu<\infty$ for all $E\in\C$. Assume that the map $\omega\mapsto J_{\omega}$ is continuous with respect to the operator norm on $J_\omega$. Then $E\in\rho(J_{\omega})$ if and only if $(T,B^E)$ admits dominated splitting.
		\end{theorem}
	 The original version of Marx's theorem was stated for quasiperiodic Jacobi cocycles whose base dynamics are minimal translations on the $d$-dimensional torus. Such base dynamics satisfy all the conditions stated in Theorem~\ref{t:marx}. But as the author indicated in \cite[Remark 1.2.ii.]{marx}, the proof works for a more general dynamics as stated in Theorem~\ref{t:marx}. The author also mentioned in \cite[Remark 1.2.iii.]{marx} that he believes the results should hold true for merely minimal base dynamics. We wish to point out that there are Johnson theorems for CMV matrices. Regarding Johnson's theorem for the standard CMV matrices, which may be considered as the unitary analog of Schr\"odinger operators, we refer the readers to \cite[Theorem 1.2]{damanikfillmanlukicyessen}. Regarding Johnson's theorem for generalized CMV matrices, which is the unitary analog of Jacobi operators, we refer the readers to \cite[Theorem 6.1]{fillmanongzhang}. It in particular plays a key role in studying the spectral properties of the unitary Almost Mathieu operator in \cite{fillmanongzhang} which arises natually in modeling a one-dimensional quantum walk whose coins are distributed quasi-periodically.
	
		However, as we explained in Section~\ref{ss:johnson} and similar to the case of Schr\"odinger operators, there should be a Johnson's theorem for Jacobi operators defined by any sequences $a\in\ell^\infty(\Z)$ and  $b\in\ell^{\infty}(\Z,\R)$. Using such a theorem, one can already deduce some partial Johnson type of results for any base dynamics. Moreover, once one has the sequence version, it should not be far away to get a full version of Johnson's theorem for Jacobi operators like the Schr\"odinger case. In other words, topological transitivity should be sufficient. In particular, Propostion~\ref{p:dense_orbit_spectrum} can be easily seen to hold true for Jacobi operators as well (see e.g. Proposition~\ref{p:dense_orbit_spectrum_J} of Section~\ref{s:dynamicalVersion}). Hence, one should be able to control all the necessary constants associated with conditions of Definition~\ref{d.uhsequence} to make them uniform. 
		
		These are the main goals of this paper: to establish thorough versions of the Johnson's theorem for Jacobi operators, both for sequence-defined and dynamically defined operators.
		
		In the case of $\SL(2,\R)$ matrices, the definition of uniform hyperbolicity for cocycles defined over base dynamics and for $\SL(2,\R)$-sequences are almost identical, see \cite{zhang2}. However, though the definition of dominated splitting for $\M(2,\C)$-cocycles over base dynamics exists (see. e.g. Definition~\ref{d:domination_dynamical} of Section~\ref{s:dynamicalVersion}), it's not immediately clear what the one for $\M(2,\C)$-sequences should be. A good such definition should at least satisfy the following conditions:
		
		\begin{enumerate}[label=(\alph*)]
			\item It generalizes the definition of uniform hyperbolicity for $\SL(2,\R)$-sequences.
			\item If a dynamically defined cocycle $(T,B)$ admits dominated splitting, then 
			$$
			B(T^{(\cdot)}\omega):\Z\to\M(2,\C)
			$$ should be a sequence that admits dominated splitting for each $\omega\in\Omega$.
			\item It should be stable under $\|\cdot\|_\infty$-perturbation.
			\item It should be the right definition for one to show a Johnson's theorem for the Jacobi operator $J_{a,b}$ with any choice of $a\in\ell^\infty(\Z)$ and  $b\in\ell^{\infty}(\Z,\R)$.
			\end{enumerate}
		
		Thus, our first task is to find such a definition. Throughout the paper, we consider cocycle maps $B$ on $\Z$ or on a compact metric space $\Omega$ to be in $\ell^\infty$. We also let $\|\cdot\|_\infty$ denotes the superemum norm in various scenarios. In particular, we may let $M>0$ to be a universal upbound for all the cocycle maps $B$, that is $\|B\|_\infty<M$. We define
	\beq\label{eq:cocycleiteration}
	B_n(j)=\begin{cases}B(j+n-1)\cdots B(j), & n\ge1,\\ I_2 , & n=0,
	\end{cases}
	\eeq 
	where $I_2$ is the identity matrix. We also define
	$$
	B_{-n}(j)= [B_{n}(j-n)]^{-1}=B(j-n)^{-1}\cdots B(j-1)^{-1}, \ n\ge 1,
	$$
	if all matrices involved are invertible. It turns out that the following definition of domination for $\M(2,\C)$-sequences serves our purposes perfectly:
		\begin{defi}\label{d.uhsequence}
			We say that $B\in\ell^{\infty}(\Z,\M(2,\C))$ admits \emph{dominated splitting} $(\CD\CS)$ if for each $j\in\Z$, there are one-dimensional spaces $E^u(j)$ and $E^s(j)$ of $\C^2$ with the following properties.
			
			\begin{enumerate}
				\item $E^u,E^s$ are $B$--invariant in the sense that for all $j\in\Z$, it holds that
				$$
				B(j)[E^u(j)] \subseteq E^u(j+1) \mbox{ and } B(j)[E^s(j)]\subseteq E^s(j+1).
				$$
				\item There exist $N \in \Z_{+}$ and $\lambda>1$ such that 
				$$
				\|B_{N}(j)\vec u(j)\|>\lambda \|B_N(j)\vec s(j)\|
				$$
				for all $j\in\Z$ and all unit vectors $\vec u(j)\in E^u(j)$ and $\vec s(j)\in E^s(j)$.
				\item $\textit{There exists } \delta >0$ such that $d\big(E^u(j), E^s(j)\big)> \delta$ for all $j \in \Z$.
				\item Let $N\in\Z_+$ be from condition (2). Then it holds $\inf_{j\in\Z}\|B_N(j)\|>0$.
			\end{enumerate} 
		\end{defi}
		Clearly, condition (2) implies that $E^s(j)\neq E^u(j)$, hence $\C^2=E^s(j)\oplus E^u(j)$, for all $j\in\Z$. In condition (3), $d(W,V)$ denotes a distance between two one-dimensional spaces $W$ and $V$ of $
		\C^2$. For its definition, see equation \eqref{eq:DistofComplexLine} at the beginning of Section~\ref{ss:uhsequence}. Here we say the space $E^u$ dominates the space $E^s$. Throughout this paper, without loss of generality, we set $\lambda =2$ in condition (2) in Definition~\ref{d.uhsequence}. From now on, $B\in\CD\CS$ means $B$ admits dominated splitting. 
		
		\begin{remark}\label{r:domination_sequence}
				Conditions (1) and (2) of Definition~\ref{d.uhsequence} are standard conditions when one defines $\mathrm{M}(2,\C)$-valued cocycles defined over a base dynamics that admit dominated splitting, see again Definition~\ref{d:domination_dynamical} of Section~\ref{s:dynamicalVersion}. We also notice that in Definition~\ref{d:domination_dynamical}, condition (3) is automatically satisfied. However, when it comes to $\mathrm{M}(2,\C)$ sequences, neither conditions (1) and (2) imply conditions (3) and (4), nor conditions (3) and (4) imply each other under conditions (1) and (2). But conditions (3) and (4) are necessary to guarantee the stability of the domination. Indeed, we have two examples where the first example satisfies conditions (1)-(3) while condition (4) fails. In the second example, we have conditions (1), (2), and (4) while condition (3) fails. Moreover, our examples are actually $\mathrm{GL}(2,\C)$-valued.
				
				\emph{Example one}. We  defined $B$ as follows:
				$$
				B(j)=\begin{pmatrix} 2^{-|j|} & 0\\\ 0 & 2^{-|j|-1}\end{pmatrix} \mbox{ for all }j\in\Z.
				$$ 
				It is clear that for this $B$ we have $E^s(j)=\mathrm{span}\{\binom{0}{1}\}$ and $E^u(j)=\mathrm{span}\{\binom{1}{0}\}$ where $E^u$ dominates $E^s$ at step $1$. By definition of distance between complex lines at the begining of Section~\ref{ss:uhsequence}, specifically, equations \eqref{eq:DistofComplexLine} and \eqref{eq:sphere_dist_vform}, we can easily see $d(E^s(j),E^u(j))=2$ for all $j\in\Z$. Evidently, condition (4) fails since it is clear that
				$$
				\inf_{j\in\Z}\|B_n(j)\|=0\mbox{ for all }n\in\Z_+.
				$$
				
				\emph{Example two}. We define $B$ as follows.
				$$
				\Lambda(j)=\begin{pmatrix} 2^{2-|j|} & 0\\\ 0 & 2^{-|j|}\end{pmatrix},\ 
				D(j)=\begin{pmatrix}  1& 1\\ 0 & 2^{-|j|}\end{pmatrix},\mbox{ and }
				$$   
				$$
				B(j):=D(j+1)\Lambda(j) D(j)^{-1} = \begin{pmatrix} 2^{2-|j|} & -3\\ 0& 2^{-|j+1|}\end{pmatrix}.
				$$
				It's clear from the construction that the first and second column vectors of $D(j)$ generates $E^u(j)$ and $E^s(j)$ of $B(j)$, respectively. It's simple calculation to see $E^u(j)$ dominates $E^s(j)$ at step 1 with the correponding $\l>2$. Moreover, it is clear that $\inf_{j\in\Z}\|B(j)\|>3$ which is nothing other condition (4) of Definition ~\ref{d.uhsequence}. However, by \eqref{eq:DistofComplexLine} and \eqref{eq:sphere_dist_vform}, it is clear that
				$$
				d(E^s(j),E^u(j))<2^{-|j|}\to 0 \mbox{ as } j\to\pm\infty.
				$$
				 For the examples above, the domination can easily be perturbed away by an arbitrarily small $\|\cdot\|_\infty$-perturbation.
				 
				 Finally, we notice that under conditions (1)-(3), condition (4) is equivalent to 
				 \[
				 \inf_{j\in\Z}\|B_n(j)\|>0\mbox{ for all } n\in\Z_+.\label{eq:equiv_c4} \tag*{(4)'}
				 \]
				 Indeed, \ref{eq:equiv_c4} clearly implies condition (4). On the other hand, conditions (1)-(3) implies that if we define $D:\Z\to\mathrm{GL}(2,\C)$ such that the first and second column vectors are unit vectors in $E^u(j)$ and $E^s(j)$, respectively, then $\inf_{j\in\Z}|\det(D(j))|>0$, $\|D\|_\infty\le 2$, and 
				 $$
				 D(j+1)^{-1}B(j)D(j):= \begin{pmatrix} \l^+_j & 0\\ 0& \l^-_j\end{pmatrix}
				 $$
                for some $\lambda^+_j$ and $\lambda^-_j$ in $\C$. It clearly implies
                \beq\label{eq:nstep_conj}
                D(j+N)^{-1}B_n(j)D(j):= \begin{pmatrix} \prod^{j+N-1}_{k=j}\l^+_k& 0\\ 0& \prod^{j+N-1}_{k=j}\l^-_k\end{pmatrix}.
                \eeq
                By condition (2), we must have for all $j\in\Z$:
                $$
                \bigg|\prod^{j+N-1}_{k=j}\l^+_k\bigg|>2\bigg|\prod^{j+N-1}_{k=j}\l^-_k\bigg|.
                $$
                By the properties of $D$ we mentioned above and by condition (4), we must have 
                $$
                \inf_{j\in\Z}\|D(j+N)^{-1}B_N(j)D(j)\|>0.
                $$
                Combine the two estimates above, we obtain
                $$
                \inf_{j\in\Z}\bigg|\prod^{j+N-1}_{k=j}\l^+_k\bigg|>0
                $$
                which in turn implies $\inf_{j\in\Z}|\l^+_j|>0$ since $\sup_{j\in\Z}|\l^+_j|<C$ for some $C>0$. Hence 
                $$
                \inf_{j\in\Z}\bigg|\prod^{j+n-1}_{k=j}\l^+_k\bigg|>0\mbox{ for all }n\ge 1.
                $$
                By \eqref{eq:nstep_conj} and properties of $D$, we then obtain \ref{eq:equiv_c4}. It turns out that often times in the remaining part of the paper, to obtain condition (4) of Definition~\ref{d.uhsequence}, we instead prove the stronger condition \ref{eq:equiv_c4} above.
				\end{remark}

		Once we have the right definition, then the main theorem of this paper can be formulated as follows. For simplicity, we also use $(T,B)\in\CD\CS$ for $B:\Omega\to\M(2,\C)$ that admits dominated splitting. For its definition, see again Definition~\ref{d:domination_dynamical} of Section~\ref{s:dynamicalVersion}. Let $B^E:\Z\to\M(2,\C)$ be the Jacobi cocycle map for the Jacobi operator $J_{a,b}$. Then
		
     \begin{theorem}\label{t.uniformhers}
			$\rho(J_{a,b})=\{E\in\C: B^{E}\in\CD\CS\}.$
		\end{theorem}
	Clearly, Theorem~\ref{t.uniformhers} generalizes Theorem~\ref{t:johnson_sequence}. Similar to the Schr\"odinger case, we have an immediate corollary of Theorem~\ref{t.uniformhers} since we have the condition (b) stated before Definition~\ref{d.uhsequence}:
	
	\begin{corollary}\label{c:jacobi_johnson}
			Consider the family of dynamically defined Jacobi operators $J_\omega$, $\omega\in\Omega$, defined by \eqref{eq:dynamical_jacobi}. If $(T,B^E)\in\CD\CS$, then $E\in \bigcap_{\omega\in\Omega}\rho(\omega).$
		\end{corollary}
However, similar to the discussion in Section~\ref{ss:johnson}, our proof of Theorem~\ref{t.uniformhers} actually implies Theorem~\ref{t.jacobi_johnson} below. Thus, the main task of this paper is to prove Theorem~\ref{t.uniformhers}. It will be divided into two directions. The first direction is:
	\beq\label{eq:firstdirection}
		\{E\in\C: B^{E}\in\CD\CS\}\subset \rho(J_{a,b})
	\eeq
	which relies on the stability of domination for $\M(2,\C)$-sequences: 
	
		\begin{theorem}\label{l.opennessDS}
			Let $B:\Z\rightarrow \mathrm{M}(2,\C)$ be such that $B \in \CD\CS$. There exists $\e >0$ such that if $\widetilde{B}:\Z\rightarrow \mathrm{M}(2,\C)$ satisfies $\| \widetilde{B}-B\|_{\infty}< \e$, then $\widetilde{B} \in \CD\CS$.
		\end{theorem}
It turns out that the key step to show Theorem~\ref{l.opennessDS} is to show the following result. Let $\D_r=\{z: |z|<r\}\subset\C\PP^1=\C\cup\{\infty\}$. For $A\in\M(2,\C)$, we let $A\cdot$ denote certain projectivized action on $\C\PP^1$ (see \eqref{eq:inducedProj2} at the beginning of Section~\ref{ss:uhsequence} for the definition). Then

\begin{theorem}\label{l:invConeFieldtoDS}
	Let $A:\Z\to\M(2,\C)$ be such that $\inf_{j\in\Z}\|A(j)\|>0$. If there exist $\a>\a'>0$ such that $A(j)\cdot(\D_\a)\subset \D_{\a'}$ for all $j\in\Z$, then $A\in\CD\CS$.
\end{theorem}
\begin{remark}\label{r:dependence of constants.}
Theorem~\ref{l:invConeFieldtoDS} basically says that the existence of an invariant cone field implies domination for a $\M(2,\C)$-sequence. Theorem~\ref{l:invConeFieldtoDS} also justifies the rationality of Definition~\ref{d.uhsequence} as well. We will prove Theorems \ref{l.opennessDS} and \ref{l:invConeFieldtoDS} in Section~\ref{ss:uhsequence}, which is of pure dynamical systems. In the proof, we will show the following facts. For $B\in\CD\CS$, we let $N(B)$, $\delta(B)$, and $m_{N(B)}(B)=\inf_j\|B_{N(B)}(j)\|$ to be constants that appear in conditions (2), (3), and (4) of Definition~\ref{d.uhsequence}, respectively.  Then in the proof of Theorem~\ref{l:invConeFieldtoDS}, we can show $N(A)$ and $\delta(A)$ depend only on $M$, $\a$, and $\a'$. Using this fact, we can further show that the $\e$ in Theorem~\ref{l.opennessDS} depends only on $N(B), \delta(B)$, $m_N(B)$ and $M$ which in turn implies that $N(\tilde B)$ and $\delta(\tilde B)$ depend only on $N(B)$, $\delta(B)$, $m_N(B)$, and $M$. Those facts will be used to deduce Corollary~\ref{c:stability_DS_dyna} in Section~\ref{s:dynamicalVersion}.
\end{remark}

Once Theorem~\ref{l.opennessDS} is proven, we have the main tool to prove the first direction \eqref{eq:firstdirection} of Theorem~\ref{t.uniformhers}. It will be done in Section~\ref{ss:uh_to_resolvent}. Here we mainly use the strategy of \cite{marx}. However, the fact that we have to deal with the singular operators directly poses certain obstacles. In fact, one of the main reasons that Theorem~\ref{t:marx} puts a lot of restrictions on the base dynamics is the following. In \cite{marx}, first it does everything for those $\omega$ along which $a(T^j\omega)\neq 0$ for all $j\in\Z$ so that the singular operators are bypassed. Then it uses all the properties of the base dynamics to extend results for nonsingular operators to the whole base space. Dealing with singular operators directly is one of the main challenges this paper has.

The other direction of  Theorem~\ref{t.uniformhers} is clearly:
\beq\label{eq:secondhalf}
\rho(J_{a,b})\subset \{E\in\C: B^{E}\in\CD\CS\}.
\eeq
which will be done in Section~\ref{ss:NotUHtoSpectrum}. That is, we need to show $B^E\in\CD\CS$ if $E\in\rho(J_{a,b})$. We will first show $B^E:\Z\to\M(2,\C)$ satisfies the stronger version \ref{eq:equiv_c4} of condtion (4) of Definition~\ref{d.uhsequence}. This is a completely new issue which does not exist either for Schr\"odinger operators or for dynamically defined Jacobi operators. As a consequence, we need to use an estimate bounding the derivatives of a polynomial via the supremum norm of the polynomial on certain compact sets. It is the so-called Markov inequality, see e.g. \cite{pierzchala}. Then we show condition \ref{eq:equiv_c4} via certain induction schemes. 

To construct two invariant directions with desired properties, i.e. conditions (1)-(3) of Definition~\ref{d.uhsequence}, the main work is to have some detailed analysis of the Green's function. Bascially, we need two different types of information regarding the Green's function. The first is a Combes-Thomas esitmate \cite{combesthomas} showing the exponential decaying of the Green's function. This is relatively simple as it follows from some standard computations. The other one, which is the main work of this section, is to relate the column vectors in Green's function to solutions of the eigenvalue equation $J_{a,b}\phi=E\phi$. The fact that $a_j=0$ for certain $j\in\Z$ leads to four differet main cases, each of which has three subcases. 

Once we put together the information above, one can construct the two invariant directions and show they have the desired properties. For each of the conditions (1)-(3) of Definition~\ref{d.uhsequence}, one then needs to regroup the $12$ different cases and treat them with different strategies.

Finally, we will show that all the estimates showing conditions (1)-(3) of Definition~\ref{d.uhsequence} depend only on $d(E,\sigma(J_{a,b}))$ which will eventually allow us to go from Theorem~\ref{t.uniformhers} to the following dynamics version:

\begin{theorem}\label{t.jacobi_johnson}
	Consider the dynamically defined Jacobi operators $J_\omega$, $\omega\in\Omega$, given by \eqref{eq:dynamical_jacobi}. Assume that $T$ be a topologically transitive and let $\omega_0$ be that $\overline{\mathrm{Orb}(\omega_0)}=\Omega$. Then 
	$$
	\rho(J_{\omega_0})=\{E\in\C: (T,B^E)\in\CD\CS\}.
	$$
\end{theorem}
It is clear that Theorem~\ref{t.jacobi_johnson} includes Theorems~\ref{t:johnson} and ~\ref{t:marx} as special cases. In some sense, this is perhaps the best version of a Johnson's theorem for dynamically defined Jacobi operators one can hope to obtain as the conditions it assumed are indentical to those of Theorem~\ref{t:johnson}. We will have detailed discussion of Theorem~\ref{t.jacobi_johnson} in Section~\ref{s:dynamicalVersion}.

\section*{Acknowledgements}
Section~\ref{s:dynamicalVersion} uses the stragety of \cite[Section 3.3.2]{zhang2} which was generously shared with Z. Z. by Wilhem Schlag. We are grateful to him for his sharing as well as for some useful suggestions. We are also grateful to Jake Fillman for catching various mistakes and typos in a preliminary version of this paper. We are thankful to Anton Gorodenski for bringing a few relevant references into our attention.

	\section{Domination for $\mathrm{M}(2,\C)$-Sequences and Its Stability}\label{ss:uhsequence} 
	
	Let $\C\PP^1=\C\cup\{\infty\}$ be the one-dimensional complex projective space, or the Rieman sphere. We mainly use the following projection maps from $\C^2\setminus\{\vec 0\}\to\C\PP^1$:
	\beq\label{eq:proj_to_cp1}
	\pi:\C^2\setminus\{\vec 0\}\to\C\PP^1 \mbox{ where }\pi \binom{z_1}{z_2}=\frac{z_2}{z_1}.
	\eeq
	Through this projection, each one-dimensional space in $\C^2$ can be identified with a point in $\C\PP^1$. Hence, we may view a point $z\in\C\PP^1$ as an one-dimension space $\mathrm{span}\{\binom{1}{z}\}$ of  $\C^2$. Note $\infty$ is considered to be $\mathrm{span}\{\binom{0}{1}\}$. For instance, by $\vec v\in z$ we mean $\vec v$ is a vector in the one-dimensional space $z$. In particular, we let $\vec z$ denotes a unit vector in $z$. We shall mainly use the following metric on $\C\PP^1$:
	\beq\label{eq:sphere_dist}
	d(z,z')=\begin{cases}\frac{2|z-z'|}{\sqrt{(1+|z|^2)(1+|z'|^2)}}, &z,\ z'\in\C;\\
	\frac{2}{\sqrt{1+|z|^2}}, &z'=\infty.\end{cases}
	\eeq
	It's a standard fact the metric above on the Rieman sphere is induced by the stereographic projection 
	$$
	P:\mathbb S^3=\{(x_1,x_2,x_3):x_1^2+x_2^2+x_3^2=1\}\to \C\PP^1
	$$
	so that $d(z,z')=\|P^{-1}(z)-P^{-1}(z')\|_{\R^3}$ for $z,z'\in \C\PP^1$, where
	$$
	P(x_1,x_2,x_3)=\begin{cases}\infty &\mbox{if } (x_1,x_2,x_3)= (0,0,1);\\ \frac{x_1+x_2i}{1-x_3}, &\mbox{otherwise}.\end{cases}
	$$
	Note $d(z,z')\le 2|z-z'|$ for all $z$ and $z'\in\C$. Let $\vec v$ and $\vec v'$ be two nonzero vectors in $\C^2$ . We let $(\vec v, \vec v')\in \mathrm{M}(2,\C)$ denotes the matrix whose column vectors are $\vec v$ and $\vec v'$. Then a direct computation shows that
	\beq\label{eq:sphere_dist_vform}
	d(\pi(\vec v), \pi(\vec v'))=\frac{2|\det (\vec v, \vec v')|}{\|\vec v\|\cdot \|\vec v'\|}.
	\eeq
	In particular, if $\vec v$ and $\vec v'$ are two unit vectors, then  
	\beq \label{eq:sphere_dist_vform2}
	d(\pi(\vec v), \pi(\vec v'))=2|\det (\vec v, \vec v')|.
	\eeq
	Since one dimensional space can be identified by the points in $\C\PP^1$, abusing the notation slightly, for two one-dimensional subspaces $V$ and  $W$ of $\C^2$, we define
	\beq\label{eq:DistofComplexLine}
	d(V,W):=d(\pi(\vec v),\pi(\vec w))
	\eeq
	 where $\vec v\in V$ and $\vec w\in W$ are nonzero vectors.
	
	Throughout this paper, we make the following assumptions: $C,\ c$ will be universal constants, where $C$ is large and $c$ is small.

	Let $A=\left(\begin{smallmatrix}a &b \\ c & d\end{smallmatrix}\right)\in \mathrm{M}(2,\C)$ be a nonzero matrix such that $\|A\|<M$. Under the projection $\pi$ as described in \eqref{eq:proj_to_cp1}, there is an induced projectivized map of $A$ acting on projective space $(\C\PP^1)\setminus \{\alpha\}$, where $\alpha$ is the eigenspace of the $0$ eigenvalue of $A$, if such exists. We denote the induced map by $A\cdot z$. Then a direct computation shows that
	\beq\label{eq:inducedProj2}
	A\cdot z: (\C\PP^1)\setminus\{\alpha\}\to\C\PP^1,\ A\cdot z=\frac{c+dz}{a+bz}.
	\eeq
    Recall $\vec z$ denotes a unit vector in the one-dimensional space $z\in \C\PP^1$. By \eqref{eq:sphere_dist_vform} and \eqref{eq:sphere_dist_vform2}, we have
    \begin{align}
    \nonumber d(A\cdot z, A\cdot  z')&= d(A\cdot \pi(\vec z), A\cdot \pi(\vec z'))\\
    \nonumber &=d(\pi (A\vec z), \pi (A\vec z'))\\
    &=\frac{2|\det (A\vec z, A\vec z')|}{\|A\vec z\|\cdot \|A\vec z'\|}\\
   \nonumber  &=\frac{2|\det (A)|}{\|A\vec z\|\cdot \|A\vec z'\|}|\det(\vec z, \vec z')|\\
    \nonumber &=\frac{2|\det (A)|}{\|A\vec z\|\cdot \|A\vec z'\|}|d(z,z')|
    \end{align}
   If $|\det(A)|>\delta>0$, then it holds that
   $$
   \inf_{\|\vec v\|=1}\|A\vec v\|=\frac{|\det(A)|}{\|A\|}>\frac\delta{M}
   $$
   which in turn implies that (change $M$ to $2M$ if necessary)
   \beq\label{eq:boundedDerivatives}
   \frac\delta{M^2}d(z,z')\le d(A\cdot z, A\cdot z')\le \frac{M^4}{\delta^2}d(z,z')\mbox{ for all }z, z'\in\C\PP^1.
   \eeq
	Sometimes, we may need to use an another projection
	$$
	\bar \pi:\C^2\setminus\{\vec 0\}\to\C\PP^1 \mbox{ where }\bar \pi \binom{z_1}{z_2}=\frac{z_1}{z_2}.
	$$
Then the projectivized map of $A$ under $\bar \pi$, denoted by $\bar A$, becomes the well-known M\"obius transformation:
	\beq\label{eq:inducedProj3}
	\bar A(z)=\frac{az+b}{cz+d}.
	\eeq
All the computation above still hold true for $\bar A$. In fact, $d(\bar A(z), \bar A(z') )=d(A\cdot z, A\cdot z')$.

\subsection{Invariant cone field implies domination: proof of Theorem~\ref{l:invConeFieldtoDS}}

In this section, we will prove Theorem~\ref{l.opennessDS}. We begin with the following lemma.

\begin{lemma}\label{l.schwartz}
	Let $f:\D_\a\to \D_{\alpha'}$ be a holomorphic function for some $0<\alpha'<\alpha$. Then there exists a  $0<\rho=\rho(\a,\a')<1$ such that it holds for all $z_1, z_2\in\D_\a$ that
	$$ 
	\left | \frac{f(z_2)-f(z_1)}{\alpha^2-\overline{f(z_1)}{f(z_2)}}\right | < \rho\left | \frac{z_2-z_1}{\alpha^2-\overline{z_1}z_2} \right |.
	$$
\end{lemma}
\begin{proof}
	For simplicity, we set $\D=\D_1$. Define $g : \D \to \D$ to be 
	\beq\label{eq:gholom}
	g(z): = \frac{1}{\alpha'}f(\alpha z)
	\eeq
	which is a holomorphic function on the unit disc. By Schwarz–Pick theorem (see \cite{dineen}), for all $z_1, z_2 \in \D$, 
	$$\left | \frac{g(z_2)-g(z_1)}{1-\overline{g(z_1)}g(z_2)} \right | \le \left | \frac{z_2-z_1}{1-\overline{z_1}z_2} \right |.$$
	Fom (\ref{eq:gholom}), it holds for all $z\in \D_{\alpha}$ that
	$$f(z)  = \alpha' g\left(\frac{z  }{\alpha}\right).$$
	Hence, for all $z_1  , z_2   \in \D_{\alpha}  $, we have 
	\begin{align*}
		\left | \frac{f(z_2)-f(z_1)}{\alpha'^2-\overline{f(z_1)}{f(z_2)}}\right |
		&= \left | \frac{\alpha'g\left(\frac{z_1}{\alpha}\right)-\alpha'g\left(\frac{z_2}{\alpha}\right)}{\alpha'^2-\overline{\alpha'g\left(\frac{z_1}{\alpha}\right)}\alpha'g\left(\frac{z_2}{\alpha}\right)} \right |\\
		&= \frac{1}{\alpha'}\left | \frac{g\left(\frac{z_1}{\alpha}\right)-g\left(\frac{z_2}{\alpha}\right)}{1-\overline{g\left(\frac{z_1}{\alpha}\right)}g\left(\frac{z_2}{\alpha}\right)} \right | \\
		& \le \frac{1}{\alpha'}\left | \frac{ \frac{z_2}{\alpha}- \frac{z_1}{\alpha}}{1-\overline{ \frac{z_1}{\alpha}} \frac{z_2}{\alpha}} \right | \\
		&=\frac{\alpha}{\alpha'}\left | \frac{z_2-z_1}{\alpha^2-\overline{z_1}z_2} \right |.
	\end{align*}
	We can rewrite
	\begin{align*}
		\left | \frac{f   ( z_2)  -f   ( z_1)  }{\alpha'^2-\overline{f   ( z_1)  }f   ( z_2)  } \right | &= \frac{\alpha^2}{\alpha'^2}\left | \frac{f   ( z_2)  -f   ( z_1)  }{\alpha^2-(\frac{\alpha}{\alpha'})^2\overline{f   ( z_1 ) }f   ( z_2 ) } \right |\\
		&= \frac{\alpha^2}{\alpha'^2} \left | \frac{f   ( z_2 ) -f   ( z_1)  }{\alpha^2-\overline{f   ( z_1)  }f   ( z_2)  } \right |\left | \frac{\alpha^2-\overline{f   ( z_1 ) }f   ( z_2  ) }{\alpha^2-(\frac{\alpha}{\alpha'})^2\overline{f   ( z_1 ) }f   ( z_2 ) } \right |
	\end{align*}	
	Hence, combining the two estimates above, we obtain for all $z_1$ and $z_2$ in $\D_\alpha$ that
	\begin{align}\label{eq:splemma}
		\left | \frac{f(z_2)-f(z_1)}{\alpha^2-\overline{f(z_1)}{f(z_2)}}\right |
		&= \left|\frac{\alpha'^2-\overline{f(z_1)}{f(z_2)}}{\alpha^2-\overline{f(z_1)}{f(z_2)}}\right|\cdot\left | \frac{f(z_2)-f(z_1)}{\alpha'^2-\overline{f(z_1)}{f(z_2)}}\right |\\
		\nonumber &\le 
		\frac{\alpha}{\alpha'}  \left|\frac{\alpha'^2-\overline{f(z_1)}{f(z_2)}}{\alpha^2-\overline{f(z_1)}{f(z_2)}}\right|  \cdot \left | \frac{z_2  -z_1  }{\alpha^2-\overline{z_1  }z_2  } \right |.
	\end{align}
	All that is left to show is that 
	\beq\label{eq:splemma1}
	\frac{\alpha}{\alpha'}  \left|\frac{\alpha'^2-\overline{f(z_1)}{f(z_2)}}{\alpha^2-\overline{f(z_1)}{f(z_2)}}\right|<\rho<1 \mbox{ for all } z_1, z_2\in\D_\a.
	\eeq
Since $|\overline{f(z_1)}f(z_2)| < \alpha'^2$ for all $z \in \D_{\alpha}$, it is clear that the estimate above can be reduced to find an upbound of the following function
	$$
	h:\D_a\to\R \mbox{ where }\ h(z)=\left|\frac{a-z}{b-z}\right| \mbox{ and }0<a<b.
	$$
	Set $z=te^{i\t}$ for some $0\le t<a$. It is then straightforward to see that 
	$$
	h(te^{i\t})=\left(\frac{a^2+t^2-2at\cos\t}{b^2+t^2-2bt\cos\t}\right)^{\frac12}
	$$
	and $\max_{\t\in[0,2\pi)}h(te^{i\t})=h(te^{i\pi})=\frac{a+t}{b+t}$ which is monotone increasing in $t\in [0,\infty)$. Hence we obtain 
	$$
	\sup_{z\in\D_a}|h(z)|\le \frac{2a}{a+b}.
	$$
	Thus, we have for the left hand side of \eqref{eq:splemma1} that
	$$
	\frac{\alpha}{\alpha'}  \left|\frac{\alpha'^2-\overline{f(z_1)}{f(z_2)}}{\alpha^2-\overline{f(z_1)}{f(z_2)}}\right|\le \frac{\alpha}{\alpha'}\frac{2\alpha'^2}{\alpha^2+\alpha'^2}=\frac{2\frac{\alpha}{\alpha'}}{1+(\frac{\alpha}{\alpha'})^2}<1 \mbox{ for all }z_1, z_2\in \D_\alpha.
	$$
	In other words, we can set $\rho=\rho(\a,\a')=\frac{2\frac{\alpha}{\alpha'}}{1+(\frac{\alpha}{\alpha'})^2}<1$.
	
\end{proof}

\begin{proof}[Proof of Theorem~\ref{l:invConeFieldtoDS}]
	Replacing $\a$ by a smaller number if necessary, we may assume that for all $j\in\Z$
	$$
	A(j)\cdot(\overline\D_\a)\subset \overline{A(j)\cdot(\D_\a)}\subset \D_{\a'}.
	$$
	By Lemma~\ref{l.schwartz}, there exists a $0<\rho=\rho(\a,\a')<1$ such that it holds for all $j\in\Z$ and all $z_1, z_2\in\D_\alpha$ that
	$$
	\left | \frac{A(j)\cdot z_2-A(j)\cdot z_1}{\alpha^2-\overline{A(j)\cdot z_1 }A(j)\cdot z_2}\right | < \rho\left | \frac{z_2-z_1}{\alpha^2-\overline{z_1}z_2} \right |.
	$$
	Thus we have for all $n\in\Z_+$, all $j\in\Z$, and all $z_1,z_2\in\overline{\D}_\alpha$ that
	$$
	\left | \frac{A_n(j)\cdot z_2-A_n(j)\cdot z_1}{\alpha^2-\overline{A_n(j)\cdot z_1 }A_n(j)\cdot z_2}\right | < \rho^{n-1}	\left | \frac{A(j)\cdot z_2-A(j)\cdot z_1}{\alpha^2-\overline{A(j)\cdot z_1 }A(j)\cdot z_2}\right |
	$$
	which in turn implies that 
	\beq\label{eq:unif_contra_dist}
	\left | A_n(j)\cdot z_2-A_n(j)\cdot z_1\right|<\frac{2\a^2\a'}{\a^2-\a'^2}\rho^{n-1}\mbox{ for all }z_1,z_2\in\D_\a
	\eeq
	Hence, the sequence of sets 
	$$
	A_n(j-n)\cdot(\overline\D_\alpha),\  n \ge 1
	$$
	are nested compact sets whose diameter goes to $0$ as $n$ goes to $\infty$. Hence, 
	$$
	\bigcap_{n \ge 1}A_n(j-n) \cdot(\overline\D_\alpha)
	$$ 
	is a single point. We may then set
	$$
	E^u(j):= \bigcap_{n \ge 1}A_n(j-n) \cdot (\overline\D_\alpha).
	$$
	One readily checks that $E^u$ is $A$- invariant. Indeed, we have
	\begin{align*}
		A(j)\cdot E^u(j)
		&=A(j)\cdot \left(\bigcap_{n \ge 1}A_n(j-n) \cdot (\overline\D_\alpha)\right)\\
		&\subset \bigcap_{n \ge 1}A(j)\cdot \left[A_n(j-n) \cdot (\overline\D_\alpha)\right]\\
		&=\bigcap_{n \ge 1}A_n(j+1-n) \cdot \left[(A(j-n)\overline\D_\alpha)\right]\\
		&\subset \bigcap_{n \ge 1}A_n(j+1-n) \cdot (\overline \D_\alpha)\\
		& =E^u(j+1),
	\end{align*}
	which is nothing other than $A(j)\cdot E^u(j)=E^u(j+1)$. Moreover, it is evident that 
	\beq\label{eq:u_inside_Da'}
	E^u(j)\in \D_{\a'}\mbox{ for all }j\in\Z.
	\eeq
	
	Next, we compute $E^s(j)$ for all $j\in\Z$. Suppose $\det A \neq 0$. Then it is clear that $A\cdot (\D_\a)\subset \D_{\a'}$ implies
	$$
	A^{-1}\cdot (\D^\complement_{\alpha'})\subset \D^\complement_{\alpha}.
	$$
	Let $g:\C\PP^1\to\C\PP^1$ be the diffeomorphism such that $g(z)=\frac1z$. Then it is clear that $g=g^{-1}$ and 
	$$
	g\circ A\cdot =\bar A\circ g.
	$$
	Also, it is evident that $g(\D^\complement_r)=\overline\D_{\frac1r}$. Thus if $\det(A(j))\neq 0$, then the estimates above imply
	$$
	\overline {A(j)^{-1}}\cdot (\overline\D_{\frac1{\alpha'}})=\overline {A(j)^{-1}}\cdot g( \D^\complement_{\alpha'})=g\circ A(j)^{-1}\cdot (\D^\complement_{\alpha'})\subset g(\D^\complement_{\alpha})=\overline\D_{\frac1\a}
	$$
	Thus we have for all $\frac1{\a'}>r>\frac1\a$
	$$
	\overline{A(j)^{-1}}\left(\overline\D_\frac{1}{\a'}\right)\subset \overline\D_{r}.
	$$
	
	Now we fix an arbitrary $j_0\in\Z$. To find $E^s(j_0)$, we consider two different cases:
	\vskip .1cm
	
	\noindent Case I. $\det A(j)\neq 0$ for all $j\ge j_0$, then similar to the argument where we obtained $E^u$, we have that
	$$
	\left\{\overline{A_{-n}(j_0+n)}\cdot\big(\overline\D_{\frac1{\a'}}\big)\right\}_{n\ge 1}=\left\{\overline{A_{n}(j_0)^{-1}}\cdot\big(\overline\D_{\frac1{\a'}}\big)\right\}_{n\ge 1}
	$$
	is a nested sequence of compact sets whose diameters tend to zero as $n\to\infty$. Thus we can define
	$$
	E^s(j_0):=g\left[\bigcap_{n \ge 1}\overline{A_{-n}(j_0+n)}\cdot\big(\overline\D_{\frac1{\a'}}\big)\right].
	$$
	Note that $E^s(j)$ can be defined in the same way for all $j\ge j_0$ in this case. Now similar to the $E^{u}$, one readily checks for all $j\ge j_0$ that
	$$
	A(j)\cdot E^s(j)=E^s(j+1)\mbox{ and }E^{s}(j)\in g\left(\bigcap_{r>\frac1\a}\D_{r}\right)=g\big(\overline\D_{\frac1\a}\big)=\D^\complement_\a.
	$$
	In particular, the conclusion above hold true for $j=j_0$.
	\vskip .1cm
	
	\noindent Case II. If $\det A(j)=0$ for some $j \ge j_0$, then we set $j'=\min\{j\ge j_0: \det A(j)=0\}$ and define
	$$
	E^s(j):= \ker (A_{j'-j_0+1}(j_0)).
	$$
	Note that our condition clearly implies that $\D_\a$ is contained in the domain of $A_n(j)$ for all $j$ and all $n\ge 1$. Thus $A_n(j)$ cannot be a zero matrix which implies that $E^s(j)$ is a one-dimensional subspace of $\C^2$. Moreover, since $A_{j'-j_0}(j_0)$ cannot be defined at $\pi(E^s(j))$, it cannot be in the domain of $A_{j'-j_0}(j_0)$. In particular, it holds that
	$$
	E^s(j_0)\in \D^\complement_\a.
	$$
	To show the invariance, we have to further consider two subcases. First, we assume $j'=j_0$. Then 
	$$
	E^s(j_0)=\ker{A(j_0)}\mbox{ and }A(j_0)[E^s(j_0)]=\{\vec 0\}\subset E^s(j_0+1)
	$$
	no matter what is $E^s(j_0+1)$. Second, we consider $j'>j_0$. Then $j'\ge j_0+1$ which implies $j_0+1$ falls into in the present second case as well. Thus, by our definition, we have 
	$$
	E^s(j_0+1)=\ker (A_{j'-j_0}(j_0+1)).
	$$
	Moreover, we have $\det A(j_0)\neq 0$ in this case. Thus we must have 
	\begin{align*}
		E^s(j_0)
		&=\ker \big[A_{j'-j_0+1}(j_0)\big]\\
		&=\ker \big[A_{j'-j_0}(j_0+1)A(j_0)\big]\\
		&=A(j_0)^{-1}[\ker (A_{j'-j_0}(j_0+1))]\\
		&=A(j_0)^{-1}[E^s(j_0+1)]
	\end{align*}
	or equivalently
	$$
	A(j_0)[E^s(j)]=E^s(j_0+1).
	$$
	
	Combining cases I and II, we have defined $E^s(j)$ for all $j\in\Z$, showed its invariance, and obtained $E^s(j)\in\D^\complement_\a$. Thus we have
	$$
	d(E^s(j),E^u(j))\ge \inf\{d(z,z'): z\in \D^\complement_{\a}, z'\in\D_{\a'}\}
	$$
	It is straightforward calculation that for all $z\in \D^\complement_{\a}$ and all $z'\in\D_{\a'}$ that
	\begin{align}
		\nonumber d(z,z')&=\frac{2|z-z'|}{\sqrt{(1+|z|^2)(1+|z'|^2)}}\ge \frac{2(|z|-\a')}{\sqrt{(1+|z|^2)(1+\a'^2)}}\\
		\label{eq:uniform_dist}&= \frac{2(1-\frac{\a'}{|z|})}{\sqrt{(1+|z|^{-2})(1+\a'^2)}}\ge \frac{2(1-\frac{\a'}{\a})}{\sqrt{(1+\a^{-2})(1+\a'^2)}}\\
		\nonumber &= \frac{2(1-\frac{\a'}{\a})}{\sqrt{(1+\a^{-2})(1+\a'^2)}}= \frac{2(\a-\a')}{\sqrt{(1+\a^{2})(1+\a'^2)}}.
	\end{align}
	Let $\delta=\frac{2(\a-\a')}{\sqrt{(1+\a^{2})(1+\a'^2)}}>0$. Thus, we have 
	\beq\label{eq:unif_separation_us}
	\inf_{j\in\Z}d(E^s(j),E^u(j))>\delta.
	\eeq
	Let $\vec u(j)\in E^u(j)$ and $\vec s(j)\in E^s(j)$ be unit vectors. We set $D(j) = \big(\vec u(j),\ \vec s(j)\big)\in \M(2,\C)$. In other words, $\vec u(j)$ and $\vec s(j)$ are the column vectors of $D(j)$.  By \eqref{eq:sphere_dist_vform2} and \eqref{eq:unif_separation_us}, we have
	\beq\label{eq:unifPosdetD}
	\inf_{j\in\Z}|\det D(j)|=\inf_{j\in\Z}d(E^u(j),E^s(j))>\delta.
	\eeq
	Clearly, $D(j)$ are invertible for all $j \in \Z$ and there are $c,C$, depending on $\delta$, such that 
	\beq\label{eq:unifNormBoundD}
	c<\|D(j)^\pm\|<C\mbox{ for all }j\in\Z.
	\eeq
 Together with $|\det(D(j))|\le \|D(j)\|^2\le 1$, \eqref{eq:unifNormBoundD} implies
	\beq\label{eq:bounddetD}
	1\le |\det D(j)^{-1}|<C\mbox{ for all }j\in\Z.
	\eeq
	Thus, we my apply \eqref{eq:boundedDerivatives} to $D(j)$ and obtain for some $c, C$, depending on $\delta$ only, such that
	\beq\label{eq:bounded_dD}
	cd(z,z')<d(D(j)^{-1}\cdot z,D(j)^{-1}\cdot z)<Cd(z,z')\mbox{ for all }j\in\Z \mbox{ and all } z, z'\in\C\PP^1.
	\eeq
    We define
	$$
	\Lambda(j):=D(j+1)^{-1}A(j)D(j).
	$$
	Since $\Lambda(j)$ leaves one dimensional spaces correspond to both $0$ and $\infty\in\C\PP^1$ invariant, it must be diagonal. In other words, there are $\lambda_j^+$ and $\l^-_j$ for all $j\in\Z$ such that
	$$
	\Lambda(j)=\begin{pmatrix} \l_j^+ & 0\\\ 0 &\l_j^-\end{pmatrix}.
	$$
	\eqref{eq:unifNormBoundD} together with the fact $0<m_1<\|A(j)\|<M$ for all $j\in\Z$ imply that
	\beq\label{eq:bound_norm_lambda}
	c_1<\|\Lambda(j)\|<C_1 \mbox{ for all }j\in\Z,
	\eeq
where $c_1, C_1$ depend on $\a, \a', m_1$ and $M$. Set $\CF(j) = D(j)^{-1}\D_\alpha$. The computation \eqref{eq:uniform_dist} shows that $\inf_{j\in\Z}d(E^u(j), \partial\D_\a)>\delta$ where $\partial \D_\a$ is the boundary of $\D_\a$. Thus \eqref{eq:bounded_dD} implies that
 \beq
 \inf_{j\in\Z}d(0,\partial\CF(j))=\inf_{j\in\Z}d\big(D(j)^{-1}\cdot E^u(j),D(j)^{-1}\cdot \partial \D_\a\big)>c\delta>0
 \eeq
In particular, there is a $r'>0$, depending on $\delta$ only, such that
$$
\D_{r'}\subset\CF(j)\mbox{ for all }j\in\Z.
$$ 
Let $\mathrm{diam}(S)$ denotes the diameter of a set $S\subset \C\PP^1$ under the metric $d$. By \eqref{eq:unif_contra_dist}, the fact $d(z,z')<2|z-z'|$ for all $z,z'\in\C$, and \eqref{eq:bounded_dD}, we obtain
\begin{align*}
	\mathrm{diam}\big(\Lambda_n(j)\CF(j) \big)&= \mathrm{diam}\big(D^{-1}(j)A_n(j)\cdot\D_\a\big) \\
&\le C\mathrm{diam}\big(A_n(j)\cdot\D_\a\big)\\
	&\le C'\rho^{n-1}
\end{align*}
where $0<\rho<1$ and $C'$ depend only on $\delta$. Note by our construction, $E^u(j)\in A_{n}(j-n)\cdot \D_\a$ for all $n\ge 1$ and all $j\in\Z$ which implies that $0\in \Lambda_n(j)\CF(j)$ for all $n\ge 1$ and $j\in\Z$. Thus the estimate above shows that $\Lambda_n(j)\CF(j)$ uniformly converges to $0$ as $n\to\infty$. In particular, there exists a $N\ge 1$, depending on $\delta$ only, such that for all $j\in \Z$
$$
\Lambda_N(j)\cdot \big(\D_{r'}\big)\subset \D_{\frac{r'}2}
$$
where
$$
\Lambda_N(j)=\begin{pmatrix}\prod_{k=j}^{j+N-1}\l_k^+ &0\\0 &\prod_{k=j}^{j+N-1}\l_k^- \end{pmatrix}
$$
Hence, we must have for all $j\in\Z$ that
$$
\left |\prod_{k=j}^{j+N-1}\l_k^+\right |> 2\left |\prod_{k=j}^{j+N-1}\l_k^- \right |.
$$
Since $A_N(j)D(j)=D(j+N)\Lambda_N(j)$ where $D(j)=(\vec u(j), \vec s(j))$, the estimate above implies that
$$
\|A_N(j)\vec u(j)\|>2\|A_N(j) \vec s(j)\|\mbox{ for all }j\in\Z.
$$
Recall that $\vec u(j)$ and $\vec s(j)$ are unit vectors in $E^u(j)$ and $E^s(j)$, respectively. Thus we have showed that $E^u$ dominates $E^s$. Moreover, the corresponding constant $\delta(A)=\delta$ and $N(A)=N$ depend on $\delta$, hence on $\a$ and $\a'$, only.

To complete the proof that $A\in\CD\CS$, the only thing left is to show for all $n\ge 1$, $\inf_{j\in\Z}\|A_n(j)\|>0$. Similar to the argument contained in Remark~\ref{r:domination_sequence} and since $A$ is conjugate to $\Lambda$ via $D$ which satisfies \eqref{eq:unifNormBoundD}, we only need to show such condition holds true for $\Lambda$. Since $\Lambda$ is diagonal, it is then sufficient to show
$$
\inf_{j\in\Z}|\lambda^+_j|>0.
$$
Suppose it is not true. Then for any $\e>0$, there exists a $j\in\Z$ such that $|\l^+_{j}|<\e$. By \eqref{eq:bound_norm_lambda}, we must have $|\l^-_{j}|>c$. Then, we must have for all $t>0$ that
$$
\D_{\frac{c}\e t}\subset \Lambda(j)\cdot{\D_t}
$$
By \eqref{eq:bounded_dD} and the simple fact that $d(z,z')<|z-z'|<C(r)d(z,z')$ for all $z,z'\in\D_r$, there are $0<\a_1<\a_2$ such that $\D_{\alpha_1}\subset \CF(j)=D(j)^{-1}\D_\a\subset \D_{\a_2}$ for all $j\in\Z$. Hence, we obtain
$$
\D_{\frac{c}\e\a_1}\subset\Lambda(j)\cdot \D_{\a_1}\subset \Lambda(j)\cdot \CF(j) \subset\CF(j+1)
\subset \D_{\a_2}
$$
which is clearly not true when $\e$ is sufficiently small. This completes the proof.
\end{proof}

\subsection{Stability of domination: proof of Theorem~\ref{l.opennessDS}}

	\begin{lemma}\label{l.B_BN_equiv}
		Let $B \in \ell^{\infty}(\Z, \mathrm{M}(2,\C))$. Define $B^{(m,N)}(j):=B_N(jN+m)$. Then, $B \in \CD\CS$ if and only if there exists $N \in \Z_{+}$, such that $B^{(m,N)} \in \CD\CS$ for all $m=0,1,2,...,N-1$.
	\end{lemma}
	\begin{proof}
		We start by showing the \textit{only if} part. Let $B \in \CD\CS$. Fix $N \in \N$ such that it satisfies condition (2) in definition ~\ref{d.uhsequence}. For simplicity, we define $A^{(m)}:\Z\rightarrow \mathrm{M}(2,\C)$ as 
		$$
		A^{(m)}(j)=B^{(m,N)}(j)=B_N(jN+m).
		$$ 
		By condition (4) of Definition~\ref{d.uhsequence}, it is clear that $\inf_{j\in\Z}\|A^{(m)}(j)\|>0$ for all $m=0,1,2,...,N-1$.

		For every $m=0,1,2,...,N-1$, define $E^{u}_{m}(j) := E^u(jN+m)$ and $E^{s}_{m}(j) := E^s(jN+m)$. Clearly $E^{s}_{m}$ and $E^{u}_{m}$ are $A^{(m)}$-invariant and they satisfy the conditions (3) of Definition~\ref{d.uhsequence}. Finally, with such choices of $E^s_m$ and $E^u_m$, by our definition of $A^{(m)}$, it is clear that $A^{(m)}$ satisfies condition (2) of Definition~\ref{d.uhsequence} at step one. With these subspaces, $A^{(m)}$ satisfies conditions (1)-(4) from the definition of $\CD\CS$. Hence, $A^{(m)}=B^{(m,N)} \in \CD \CS$.

		Now we show the \textit{if} part. Again, we define $A^{(m)}:\Z\rightarrow \mathrm{M}(2,\C)$ as 
		$$
		A^{(m)}(j)=B^{(m,N)}(j)=B_N(jN+m)
		$$
		and we have $A^{(m)} \in \CD \CS$ for all $m=0,1,2,...,N-1$. In particular, we may define $E^{u}_{m}(j)$ and $E^{s}_{m}(j)$ to be the invariant subspaces of $A^{(m)}$. Let $\vec s^{(m)}(j)$ and $\vec u^{(m)}(j)$ be unit vectors in $E^{s}_{m}(j)$ and $E^{u}_{m}(j)$, respectively. We may choose a $\delta>0$ and a $K\in\Z_+$ so that for all $0\le m<N$ and all $j\in\Z$, it holds that 
		\begin{align}
			\label{eq:sepsu_Am}&\|E^s_m(j)-E^u_m(j)\|_{\R\PP^1}>\delta \\
			\label{eq:domination_Am}&\|A^{(m)}_K(j)\vec u^{(m)}(j)\|>2\|A^{(m)}_K(j)\vec s^{(m)}(j)\|
		\end{align}
		
		For every $m \in \{0,1,2,...,N-1\}$, we can define a pair of invariant directions $E^{u,m}$ and $E^{s,m}$ for $B$ as follows. Fix any $j\in\Z$, we have $j=kN+m+\ell$ for some $k\in\Z$ and some $0\le \ell<N$. We define 
		$E^{u,m}(j)$ to be
		$$
		E^{u,m}(j): = B_\ell(j-\ell)E^u_m(k)
		$$
		which is always a one-dimesional subspace of $\C^2$. Otherwise, we will have $A^{m}(k)E^u_m(k)=\{\vec 0\}$ which contradicts the condition (2) of Definition~\ref{d.uhsequence}. We define $E^{s,m}(j)$ to be the one-dimensional subspace of $\C^2$ so that
		$$
		B_{N-\ell}(j)E^{s,m}(j)\subset E^s_m(k+1).
		$$
		That is, if  $B_{N-\ell}(j)$ is invertible, then $E^{s,m}(j)=B_{N-\ell}(j)^{-1}E^s_m(k+1)$; otherwise, $E^{s,m}(j)$ is the eigenspace of  $B_{N-\ell}(j)$ for the eigenvalue $0$.
		
		One readily checks that $E^{u,m}(j)$ and $E^{s,m}(j)$ are $B$-invariant for each $0\le m<N$. Moreover, $E^{s,m}(j)\neq E^{u,m}(j)$ for all $j\in\Z$. We now have $N$ pairs of invariant subspaces for $B(j)$.  We claim that the $N$ choices must be the same. To this end, we first note that it holds that for all $k\in\Z$ that
		$$
		E^{u,m}(kN+m)=E^u_m(k)\mbox{ and }E^{s,m}(kN+m)=E^s_m(k).
		$$
		We define $\vec s_m(j)\in E^{s,m}(j)$ and $\vec u_m(j)\in E^{u,m}(j)$ to be unit vectors. Fix any $j\in\Z$, we may assume $j=kN+m$ for some $k\in\Z$ and some $0\le m<N$. It then suffices to show that for all $0\le\ell<N$, $\ell\neq m$, we have
		$$
		E^{s,\ell}(j) =E^{s,m}(j)\mbox{ and }E^{u,\ell}(j)=E^{u,m}(j).
		$$
		We assume that $\ell>m$ since the proof of the case $\ell<m$ is completely analogous. We write $\vec s_\ell(j)=a\vec u_m(j)+b\vec s_m(j)$ and $\vec u_\ell(j)=c\vec u_m(j)+d\vec s_m(j)$ for some $a, b, c, d\in\C$. It is clear that $\vec u_m(j)\in E^u_m(k)$ and $\vec s_m(j)\in E^s_m(k)$. Recall that $E^u_m$ dominates $E^s_m$ under $A^{(m)}=B_N(N(\cdot)+m)$. Thus for all $\lambda>0$, we have for some large $n_1=pN\in\Z_+$ that 
		\beq\label{eq:samedirection1}
		\|B_{n_1}(j)\vec u_m(j)\|>\lambda\|B_{n_1}(j)\vec s_m(j)\|.
		\eeq
		Note we have $B_{n_1}(j)\vec u_m(j)\in E^u_m(p+k)$ and $B_{n_1}(j)\vec s_m(j)\in E^s_m(p+k)$ .
		
		On the other hand, $B_{\ell-m}(j)\vec s_\ell(j)\in E^s_\ell(k)$ and $B_{\ell-m}(j)\vec u_\ell(j)\in E^u_\ell(k)$. Thus for the same $\lambda>0$ above and by choosing $p$ appropriately, we can find a $n_2=n_1+\ell-m$ such that 
		\beq\label{eq:samedirection2}
		\|B_{n_2}(j)\vec s_\ell(j)\|<\frac1\lambda\|B_{n_2}(j)\vec u_\ell(j)\|
		\eeq
		For simplicity, we set $D=B_{\ell-m}(j+n_1)$.  \eqref{eq:samedirection2} implies that 
		\begin{align*}
			&|a|\|DB_{n_1}(j)\vec u_m(j)\|-|b|\|DB_{n_1}(j)\vec s_m(j)\|\\
			\le &\|DB_{n_1}(j)\vec s_\ell(j)\|\\
			=&\|B_{n_2}(j)\vec s_\ell(j)\|\\
			<&\frac1\lambda \|B_{n_2}(j)\vec u_\ell(j)\|\\
			=&\frac1\lambda \|DB_{n_1}(j)\vec u_\ell(j)\|\\
			\le & \frac{|c|}\lambda\|DB_{n_1}(j)\vec u_m(j)\|+\frac{|d|}\lambda\|DB_{n_1}(j)\vec s_m(j)\|
		\end{align*}
		If $a\neq 0$, then for large $\lambda$ the estimate above and \eqref{eq:samedirection1} imply that
		\begin{align*}
			\|DB_{n_1}(j)\vec u_m(j)\|&<\frac{2(|b|+|d|)}{|a|}\|DB_{n_1}(j)\vec s_m(j)\|\\
			&\le \frac{2(|b|+|d|)\|D\|}{|a|}\|B_{n_1}(j)\vec s_m(j)\|\\
			&\le \frac{2(|b|+|d|)\|D\|}{\lambda |a|}\|B_{n_1}(j)\vec u_m(j)\|
		\end{align*}
		which in turn implies for some $C_1>0$, the choice of which is independent of $\lambda$, that
		\beq\label{eq:samedirection3}
		\|B_{\ell-m}(j+n_1)\vec u^{(m)}(k+p)\|\le \frac{C_1}{\lambda}
		\eeq
		where $\vec u^{(m)}(k+p)$ is a unit vector in $E^u_m(p+k)$. Note that \eqref{eq:samedirection3} implies for some $C_2>0$, the choice of which is independent of $\lambda$, that
		$$
		\|B_{KN}((k+p)N+m)\vec u^{(m)}(k+p)\|\le \frac {C_2}{\lambda}
		$$
		which together with \eqref{eq:domination_Am} implies
		$$
		\|B_{KN}((k+p)N+m)\vec s^{(m)}(k+p)\|\le \frac {C_2}{2\lambda}
		$$
		where $\vec s^{(m)}(k+p)$ is an unit vector in $E^s_m(k+p)$.  Note \eqref{eq:sepsu_Am} implies that the angle between $\vec u^{(m)}(k+p)$ and $\vec u^{(m)}(k+p)$ is at least $\delta$ away from $0$ and $\pi$. Thus, the two estimates above clearly imply that 
		$$
		\|B_{KN}((k+p)N+m)\vec w\|\le \frac {C_3}{\lambda}
		$$
		for all unit vectors $\vec w\in\R^2$. Thus, we have 
		$$
		\|B_{KN}((k+p)N+m)\|=\|A^{(m)}_K(k+p)\|<\frac{C_3}\lambda
		$$
		where the choice of $C_3$ is again independent of $\lambda$. Since we can find such $p$ for all $\lambda>0$, it contradicts with condition (4) of Definition~\ref{d.uhsequence} for $A^{(m)}$. Thus we must have $a=0$ which implies that 
		$$
		\vec s_\ell(j)=e^{it}\vec s_m(j)\mbox{ for some }t\in\R. 
		$$
		
		Use the same strategy of proof, we can show that $d=0$ which implies that 
		$$
		\vec u_\ell(j)=e^{it}\vec u_m(j) \mbox{ for some }t\in\R. 
		$$
		Indeed, by $B$- invariance of $E^u_\ell$, we have $B_n(j-n)\vec u_\ell(j-n)$ is linearly dependent with $\vec u_\ell(j)$. If $d\neq 0$, then we must have that  $\vec u_\ell(j-n)=a_n\vec u_m(j-n)+b_n\vec s_m(j-n)$  where $b_n\neq 0$ and $B_n(j-n)s_m(j-n)\neq 0$ for all $n>0$. This implies that $B_n(j-n)$ is invertible for all $n>1$ which in turn implies that $B(j-n)$ is invertible for all $n>0$. Then by the strategy of showing $b=0$ and by going backwards in time, we can show $d=0$.
		
		Now, we just need to define $E^u$ and $E^s$ of $B$ to be $E^{s,m}$ and $E^{u,m}$ for any $0\le m<N$, respectively. They satisfy condition (1) of Definition~\ref{d.uhsequence} since they are $B$-invariant. On the other hand, for any $j\in\Z$, we have $j=kN+m$ for some $k\in\Z$ and $0\le m<N$. Thus $E^s(j)=E^s_m(k)$ and $E^u(j)=E^u_m(k)$ which implies the following two things. First,  by \eqref{eq:sepsu_Am}, we have $d(E^s(j),E^u(j))>\delta$ for all $j\in\Z$. In other words, $E^s$ and $E^u$ satisfy condition (3) of Definition~\ref{d.uhsequence}. Second, by \eqref{eq:domination_Am}, it holds for all $j\in\Z$ that
		$$
		\|B_{KN}(j)\vec u(j)\|>2\|B_{KN}(j)\vec s(j)\|
		$$
		which is the condition (2) of Definition~\ref{d.uhsequence}. Finally, condition (4) of Definition~\ref{d.uhsequence} for $A^{(m)}$, $m=0,\ldots, N-1$, imply that
		$$
		\inf_{j\in\Z}\|B_{KN}(j)\|=\inf_{0\le m<N, j\in\Z}\|A^{(m)}(j)\|>0
		$$ 
		In other words, $B$ satisfies condtion (4) of Definition~\ref{d.uhsequence}. Hence,  $B\in \CD \CS$ where the corresponding invariant spaces are $E^u$ and $E^s$. 
	\end{proof}
	\begin{remark}\label{r:equivalence_inv_directions}
		The proof of the \textit{if} part of Lemma~\ref{l.B_BN_equiv} in particular implies the uniqueness of invariant directions appear in Definition~\ref{d.uhsequence}. In other words, suppose $B\in\CD\CS$ and $E^s$ and $E^u$ are the corresponding invariant directions. Suppose $F^s$ and $F^u$ form another pair of invariant directions of $B$ that satisfy conditions (1)-(3), then we must have $E^s(j)=F^s(j)$ and $E^u(j)=F^u(j)$ for all $j\in\Z$. Moreover, the proof of $E^s(j)=F^s(j)$ only involves $B(n)$ for all $n\ge j$. Similarly, the proof of $E^u(j)=F^u(j)$ only involves $B(n)$ for all $n<j$.
		\end{remark}
	
	\begin{lemma} 
		\label{l.closenessofdir}
		Let $\Lambda = \left(\begin{smallmatrix}\lambda^+ &0\\ 0 &\lambda^-\end{smallmatrix}\right)\in \mathrm{M}(2,\C)$ be such that 
 $|\lambda^+|> \gamma>0$ and $\left|\lambda^+\right|>2|\lambda^-|$.
Fix any $r>0$. Then there exist $c=c(\gamma, r)>0$ and $C=C(\gamma,r)>0$ so that if  $\widetilde\Lambda:\Z\to\M(2,\R)$ satisfies $\|\widetilde{\Lambda}-\Lambda\|\le c$, then
			\beq \label{eq:constraction_persists}
			\sup_{z\in\D_r}|\Lambda\cdot z-\widetilde{\Lambda}\cdot z| < C\|\widetilde\Lambda-\Lambda\|.
			\eeq
	\end{lemma}

	\begin{proof}
		Let $\tilde{\Lambda} = \left(\begin{smallmatrix}a &b\\ c &d\end{smallmatrix}\right)$ and let $\delta=\|\widetilde{\Lambda}-\Lambda\|$. Then, $|\lambda^+ - a|\le  \delta$, $|\lambda^- - d|\le   \delta$, $|b|\le   \delta$ and $|c|\le  \delta$ which in turn imply that
	\begin{align*}|c +d z  - \lambda^-z |&\le |c|+|z| |d-\l^-|\le \delta(1+|z|) \mbox{ and }\\
		|a +b z  - \lambda^+|&\le |a-\l^+|+|b||z| \le \delta(1+|z |).
		\end{align*}
	In particular, we have 
	$$
	|a+bz|\ge \l^+-\delta(1+|z|)\ge \gamma-\delta(1+r)\ge \frac\gamma2
	$$
	provided $\delta<\frac\gamma{2(1+r)}$. Thus, for such $\delta$, we have for all $z\in\D_r$ that
	\begin{align*}
		|\Lambda\cdot z -\widetilde{\Lambda}\cdot z |
		&\le \left| \frac{\lambda^-z }{\lambda^+} - \frac{c +d z }{a +b z }\right|\\
		&\le \left| \frac{\lambda^-z }{\lambda^+} - \frac{\lambda^-z }{a +b z }\right|+\left| \frac{\lambda^-z }{a +b z } - \frac{c +d z }{a +b z }\right|\\
		&= \left | \frac{\lambda^-z }{\lambda^+(a +b z )}\right ||a +b z -\lambda^+|+\left| \frac{1}{a +b z }\right||\lambda^-z -(c +d z )|\\
		&\le \frac{1}{|a +b z |}\delta (1+|z|) \bigg(\left|\frac{\lambda^-_j}{\lambda^+_j}z \right|+1\bigg)\\
		&\le \frac1\gamma (1+r)(2+r)\delta\\
		&=\frac1\gamma (1+r)(2+r)\|\Lambda-\widetilde\Lambda\|.
	\end{align*}
In other words, \eqref{eq:constraction_persists} holds true with $c=\frac{\gamma}{2(1+r)}$ and $C=\frac1\gamma (1+r)(2+r)$.
\end{proof}

	We are now ready to show that our definition of the dominated splitting for $\mathrm{M}(2,\C)$-sequences is a stable property under the $\|\cdot \|_\infty$- perturbation.

	\begin{proof}[Proof of Theorem~\ref{l.opennessDS}]
		
		Let $B \in \CD\CS$. Let $E^u$ and $E^s$ be its invariant spaces. By condition (2) of Definition \ref{d.uhsequence}, there exists $N \in \Z_{+}$ such that $\|B_{N}(j)\vec u(j)\|>2\|B_N(j)\vec s(j)\|$ for all $j\in\Z$, and all unit vectors $\vec u(j)\in E^u(j),\vec s(j)\in E^s(j)$. By Lemma~\ref{l.B_BN_equiv}, we may assume $N=1$ with the understanding that $m_1=\inf_j\|B(j)\|$ appear below is essentially $m_N=\inf_j\|B_N(j)\|$. We define
		\beq \label{eq:decompose}
		\Lambda(j):=D(j+1)^{-1}B(j)D(j)
		\eeq
		where $D(j)$ is defined as 
		$$
		D(j) = \big(\vec u(j),\ \vec s(j)\big).
		$$
		By invariance of $E^u$ and $E^s$, it is clear that $\Lambda:\Z\to \mathrm{M}(2,\C)$ is diagonal, i.e. 
		$$
		\Lambda(j)=\begin{pmatrix} \lambda_j^+ & 0\\\ 0 &\lambda_j^-\end{pmatrix}.
		$$  
		Moreover, since $E^u$ dominates $E^s$ at step $1$, we must have 
		$$
		\left|\lambda_j^+\right|>2\left|\lambda_j^-\right| \mbox{ for all }j \in \Z .
		$$
		By condition (4) of Definition \ref{d.uhsequence}, 
		$$ 
        \inf_{j\in\Z}d(E^ u(j),E^s(j))>\delta>0
		$$ 
		which implies
		$$
		\inf_{j\in\Z}|\det D(j)| = d(E^ u(j),E^s(j))>\delta>0.
		$$ 
		Thus, there exist $c, C$, depending on $\delta$ only, such that 
		\beq\label{eq:boundnormD2}
		c<\|D(j)^{\pm}\|<C \mbox{ for all } j\in \Z
		\eeq
		By condition (3) of Definition \ref{d.uhsequence}, $m_1<\|B(j)\|<M$ for all $j \in \Z$.  Hence, we must have 
\beq\label{eq:boundnormLambda}
		c_1<\|\Lambda(j)\|<C_1\mbox{ for all }j \in \Z,
		\eeq
		where $c_1$ and $C_1$ depend on $m_1, M$ and $\delta$. Since $\left|\lambda_j^+\right|>2|\lambda_j^-|$ for all $j \in \Z$, there must exist a $\gamma=\gamma(m_1,M,\delta)>0$ such that $|\lambda_j^+ |>\gamma$ for all $j \in \Z$. Moreover, for any $\a>0$, we have 
		\beq\label{eq:closelambda}
\Lambda(j)\cdot (\D_\a)\subset \D_{\frac\a2}\mbox{ for all } j\in \Z.
		\eeq
	
		For $\widetilde B:\Z\to\mathrm{M}(2,\C)$, we define 
		\beq\label{eq:tildelambda}
		\widetilde{\Lambda}(j):=D(j+1)^{-1}\widetilde{B}(j)D(j).
		\eeq
		Then by \eqref{eq:boundnormD2}, we have for some $C=C(\gamma, \a)>0$ that
        $$
		\|\widetilde{\Lambda}(j)-\Lambda(j)\|< C\|\widetilde B-B\|.
		$$
		Hence, if $\|\widetilde B-B\|$ small, we have the following two properties. First, by \eqref{eq:boundnormLambda}, we have 
		\beq\label{eq:lowerbdnorm_tildeLambda}
		\inf_{j\in\Z}\|\widetilde\Lambda(j)\|>c.
		\eeq
		Second, by \eqref{eq:constraction_persists} of Lemma~\ref{l.closenessofdir}, we have 
		$$
		\sup_{z\in\D_\alpha}|\Lambda(j)\cdot z-\widetilde{\Lambda}(j)\cdot z| < C\|\widetilde\Lambda(j)-\Lambda(j)\|<C^2\|\widetilde B(j)-B(j)\|.
		$$
Thus for all $\frac\a2<\a'<\a$, there exists a $\e=\e(\delta,\a,\a',\gamma)>0$ such that if $\|\widetilde B-B\|<\e$, then
		\beq \label{eq:closureoftl}
	\widetilde{\Lambda}(j)\cdot \D_\alpha \subset \D_{\a'} \mbox{ for all } j\in \Z.
		\eeq
\eqref{eq:lowerbdnorm_tildeLambda} and \eqref{eq:closureoftl} imply that $\widetilde\Lambda:\Z\to\M(2,\C)$ satisfies all the conditions of  Theorem~\ref{l:invConeFieldtoDS}. Thus, we have $\widetilde\Lambda\in \CD\CS$ with the corresponding constants $\delta(\widetilde\Lambda)$ and $N(\widetilde\Lambda)$ depend only on $\e$, hence only on $m_{N(B)}$, $M$, $\delta(B)$, $\a$ and $\a'$. However, the choice of $\a$ and $\a'$ are independent of $B$.

Since $\widetilde\Lambda$ is conjugate to $\widetilde B$ via $D$, one readily check that $\widetilde B\in\CD\CS$. Indeed, condition (4) of Definition~\ref{d.uhsequence} for $\widetilde\Lambda$ clearly implies $\widetilde B$ satisfies condition (4) as well. Moreover, if $\tilde E^u$ and $\tilde E^s$ are the invariant spaces of $\widetilde \Lambda$, then it is straightforward to see that $\bar E^u(j)=D(j)\cdot E^u(j)$ and $\bar E^s(j)=D(j)\cdot E^s(j)$  are the invariant spaces of $\widetilde B$ which meet all other conditions of Definition~\ref{d.uhsequence}. Finally, since $\a$ and $\a'$ can be chosen independent of $B$ and since constants associated with $D$ depend only on $\delta$, one readily checks that the dependence of the constants $\delta(\widetilde B)$ and $N(\widetilde B)$ depend on $B$ only through $\delta(B)$, $N(B)$, $m_{N(B)}$, and $M$.
\end{proof}

	\section{Domination Implies Invertibility of the Operator}\label{ss:uh_to_resolvent}
	
In this section, we show that 
	$$
	\{E\in\C:B^{E} \in \CD\CS\ \}\subset\rho(J_{a,b}).
	$$
	Consider the spectral equation
	\beq\label{eq:SpecEq}
	J_{a,b}\psi=E\psi,
	\eeq
	where $E\in\C$ is the \emph{energy parameter}. A direct computation shows that if $\psi\in\C^\Z$ solves equation~\eqref{eq:SpecEq}, then
	\beq\label{eq:SchrCocycle1}
	B^{E}(j)\binom{\psi(j)}{\psi(j-1)}=a_j\binom{\psi(j+1)}{\psi(j)}\mbox{ for all } j\in\Z,
	\eeq
	where $B^{E}:\Z\to \mathrm{M}(2,\R)$ is called the \emph{Jacobi cocycle map} and is defined as
	\beq\label{eq:SchrCocylce2}
	B^{E}(j)=\begin{pmatrix}E-b_j &-\overline{a_{j-1}}\\ a_j &0\end{pmatrix}.
	\eeq
	
	In case $a_j\neq 0$, we define $A^E(j)=\frac1{a_j}B^E(j)$. By \eqref{eq:SchrCocycle1},  $(J_{a,b}\psi)_j=(E\psi)_j$ is equivalent to
	\beq\label{eq:n_step_transfer0}
	A^{E}(j)\binom{\psi(j)}{\psi(j-1)}= \binom{\psi(j+1)}{\psi(j)}.
	\eeq
	In other words, $A^E(j)$ becomes the transfer matrix of the operator~\eqref{eq:operators2} when it exists. First, we have the following lemma.
	
	\begin{lemma}\label{l.expgrowth}
		Suppose $E \in \C$ is such that $B^E \in \CD\CS$ with $E^u$ and $E^s$ being its invariant directions. Then there exists a $\l_0>1$ so that the following holds true.
		\begin{enumerate}
			\item
			Suppose there is $j_0\in\Z$ such that $a_j\neq 0$ for all $j>j_0$. Then for each $j>j_0$, there exists a $k=k(j)>0$ such that it holds for all unit vectors $\vec u(j) \in E^u(j)$ that
			\beq\label{eq:expgrowth+}
			\|A^E_{n}(j)\vec u(j) \| \ge k \lambda_0^{n} \mbox{ for all } n\ge 1,
			\eeq
			\item Suppose there is a $j_0\in\Z$ such that $a_j\neq 0$ for all $j<j_0$. Then, for all $j \le j_0$, there exists $k=k(j)>0$ such that it holds for all unit vectors $\vec s(j) \in E^s(j)$ that
			\beq\label{eq:expgrowth-}
			\|A^E_{-n}(j)\vec s(j) \| \ge k \lambda_0^{n} \mbox{ for all } n\ge1.
			\eeq
		\end{enumerate}
	\end{lemma}
	
	\begin{proof}
		By conditions (1)-(2) of Definition~\ref{d.uhsequence}, $B^E \in \CD\CS$ guarantees that there exist $c>0$ and $\l>1$ such that  
		\beq\label{eq:sl2rb}\|B^E_{n}(j)\vec u(j)\|>c\lambda^n \|B^E_{n}(j)\vec s(j)\| \eeq
		for all $j\in\Z$ and all unit vectors $\vec u(j)\in E^u(j)$ and $\vec s(j)\in E^s(j)$. 
		By condition (3) of Definition~\ref{d.uhsequence}, there exists $\delta>0$ such that $|\det \begin{pmatrix} \vec u(j) ,\vec s(j)\end{pmatrix}|=d(E^u(j),E^s(j)) >\delta>0$ for all $j \in \Z$. 
		Hence, we have
		\begin{align*}
			\delta\big| \prod_{i = j}^{j+n-1} a_{i}\overline{a_{i-1}} \big|&\le |\det (B^E_{n}(j)|d(E^u(j),E^s(j))\\
			&=|\det (B^E_{n}(j))|\cdot|\det\begin{pmatrix} \vec u(j), \vec s(j)\end{pmatrix}|\\
			&=|\det (B^E_{n}(j)\begin{pmatrix} \vec u(j), \vec s(j)\end{pmatrix})|\\
			&= |\det  \begin{pmatrix} B^E_{n}(j)\vec u(j), B^E_{n}(j)\vec s(j)\end{pmatrix} |\\
			&= \|B^E_{n}(j)\vec u(j) \| \cdot\|B^E_{n}(j)\vec s(j) \| \cdot|\det (\vec u(j+n), \vec s(j+n))|\\
		&=\|B^E_{n}(j)\vec u(j) \| \cdot\|B^E_{n}(j)\vec s(j) \| \cdot d\big(E^u(j+n), E^s(j+n)\big)\\
		&\le \delta \|B^E_{n}(j)\vec u(j) \| \cdot\|B^E_{n}(j)\vec s(j) \| .
	\end{align*}
	Combining the estimate above and \eqref{eq:sl2rb}, we have for some $\tilde c>0$, the choice of which is independent of $j$, that 
		\beq \label{eq:expgro}
		\left| \widetilde{c} \prod_{i = j}^{j+n-1} a_{i}\overline{a_{i-1}}\right|   \le  \|B^E_{n}(j)\vec u(j) \|^2  \lambda^{-n}\mbox{ for all }j\in\Z\mbox{ and all } n\ge 1.
		\eeq
		If $a_j \not =0$ for all $j > j_0$, then by \eqref{eq:expgro} it holds  for all $j>j_0+1$ that
		$$
		\|A^E_{n}(j)\vec u(j) \|^2 \ge \widetilde{c} \lambda^{n} \left| \frac{a_{j-1}}{a_{j+n-1}} \right| \ge \widetilde{c_1} \lambda^{n} |a_{j-1}| \mbox{ for all } n\ge 1
		$$
		where the last inequality comes from the fact that $a\in\ell^\infty(\Z)$. It clearly implies \eqref{eq:expgrowth+} with $\l_0=\sqrt{\l}$ when $j>j_0+1$. No matter $a_{j_0}=0$ or not, the estimate above can then be extended to $j_0+1$ since
		$$
		A^E(j_0+1)E^u(j_0+1)=E^u(j_0+2).
		$$
		
		Similarly,  suppose for some $j$, $B^E_{-n}(j)$ is well-defined for all $n\ge 1$. Then $B^E \in \CD\CS$ guarantees that there exists a $c>0$, the choice of which is independent of $j$, such that  
		$$ 
		\|B^E_{-n}(j)\vec s(j)\|>c\lambda^n \|B^E_{-n}(j)\vec u(j)\| 
		$$
		for  all unit vectors $\vec u(j)\in E^u(j)$ and $\vec s(j)\in E^s(j)$. Following the same logic as above, we get that 
		$$
		\left| \frac{\widetilde{c}}{ \prod_{i = j-n}^{j-1} a_{i}\overline{a_{i-1}}}\right| \le \|B^E_{-n}(j)\vec u(j) \| \|B^E_{-n}(j)\vec s(j) \| \le \|B^E_{-n}(j)\vec s(j) \|^2  \lambda^{-n}
		$$
	   which in turn implies for all $n\ge 1$ 
		$$
		\|A^E_{-n}(j)\vec s(j) \|^2 \ge \widetilde{c} \lambda^{n} \left| \frac{a_{j-1}}{a_{j-n-1}} \right| \ge \widetilde{c_1} \lambda^{n} |a_{j-1}|.
		$$
		If $a_j\neq 0$ for all $j<j_0$, then $B^E_{-n}(j)$ is well-defined for all $j\le j_0$ and for all $n\ge 1$. Thus, the estimate above clearly implies \eqref{eq:expgrowth-} with $\l_0=\sqrt{\l}$ for all $j\le j_0$. 
	\end{proof}

For $j_2>j_1$, we define $J_{(j_1,j_2]}$ to be the restriction of $J_{a,b}$ to the subspaces $\ell^2(j_1, j_2]$ with Dirichlet boundary conditions. That is $j_{(j_1,j_1+1]}=b_{j_1+1}$; for $j_2>j_1+1$, we have
\beq\label{eq:J_restriction1}
(J_{ (j_1,j_2]}\phi)(j) = 
\begin{cases}a_{j_1+1}\phi(j_1+2)+b_{j_1+1}\phi(j_1+1), & j=j_1+1,\\
	(J_{a,b}\phi)(j), & j_1+1<j< j_2 \mbox{ if }j_2>j_1+2,\\
	\overline{a_{j_2-1}}\phi(j_2-1)+b_{j_2}\phi(j_2), & j=j_2.\end{cases}
\eeq
We also denote restrictions of $J_{a,b}$ on half lines $(j_0,+\infty)$ and $(-\infty, j_0]$ with Dirichlet boundary conditions by  $J_{(j_0,+)}$ and $J_{(-,j_0]}$, respectively. In other words, we have
\beq\label{eq:J_restriction2}
(J_{ (j_0,+)} \phi)(j)= \begin{cases}a_{j_0+1}\phi(j_0+2)+b_{j_0+1}\phi(j_0+1), & j=j_0+1,\\
	(J_{a,b}\phi)(j), & j>j_0+1,\end{cases}
\eeq

\beq\label{eq:J-restriction3}
(J_{(-,j_0]}\phi)(j) = \begin{cases}\overline{a_{j_0-1}}\phi(j_0-1)+b_{j_0}\phi(j_0), & j=j_0,\\
	(J_{a,b}\phi)(j), & j<j_0.\end{cases}
\eeq
Note for all those restrictions, we write the intervals of integers to be half open without the left end point. They may certainly be written as other type of  intervals. We define $p_N(j, E)=\det (E - J_{[j,j+N)})$ for $N\ge 1$, $p_0(j,E)=1$, and $p_{-1}(j,E)=0$. Then it is a standard result that the following is true for all $j\in\Z$ and all $N\ge 1$:
$$
B_N^E(j) = \begin{pmatrix}
	p_N(j,E) & -\overline{a_{j-1}}p_{N-1}(j+1,E) \\
	a_{j+N-1}p_{N-1}(j,E) & -\overline{a_{j-1}}a_{j+N-1}p_{N-2}(j+1,E) 
\end{pmatrix}.
$$
	
	\begin{lemma}\label{l.trivial}
		Suppose $E \in \C$ is such that $B^E \in \CD\CS$ with $E^u$, $E^s$ being its invariant directions. We consider the following restrictions of $J_{a,b}$:
	\begin{enumerate}
	\item There is $j_0\in\Z$ such that $a_{j_0}=0$ and $a_{j_0}\neq 0$ for all $j>j_0$. Then we consider $J_{(j_0,+)}$.
	\item There is $j_0\in\Z$ such that $a_{j_0}=0$ and $a_{j_0}\neq 0$ for all $j<j_0$. Then we consider $J_{(-,j_0]}$.
	\item For all $j\in\Z$, $a_j\neq 0$. Then we consider the whole line operator $J_{a,b}$.
	\item There exist $j_1<j_2$ such that $a_{j_1}=a_{j_2}=0$. Then we consider $J_{(j_1,j_2]}$.
	\end{enumerate}
	In cases (1) and (2), the solution space of $J_{(j_0,+)} \phi = E \phi$ or $J_{(-,j_0]}\phi=E\phi$ is one-dimensional. Moreover, any of their nontrivial solutions grows exponentially along some subsequences that go to $\pm\infty$, respectively. In case (3), the solution space of $J_{a,b}\psi=E\psi$ is two dimensional and any nontrivial solution either grows exponentially along a subsequence that goes to $\infty$ or along a subsequence that goes to $-\infty$. In case (4),  $J_{(j_1,j_2]}\phi=E\phi$ has only zero solution.
	\end{lemma}
	\begin{proof}
		Without loss of generality, we may focus on case (1) since case (2) can be done similarly. Consider the solution space $\{\phi:\Z_{j>j_{0}}\to\C: J_{(j_{0},+)}\phi=E\phi\}$. By \eqref{eq:J_restriction2}, the solution is uniquely determined by $\phi_{j_0}$ which implies the space is one-dimensional. More precisely, for all $j\ge j_{0}+2$, $\phi(j_0)$ is determined by equation \eqref{eq:J_restriction2} which by \eqref{eq:n_step_transfer0} is equivalent to
	    $$
	    \binom{\phi(j+1)}{\phi(j)}=A^E_{j-j_0-1}(j_{0}+1)\binom{\phi(j_{0}+1)}{\phi(j_{0})}.
	    $$
	    Note in the equation above, the choice of $\phi(j_0)$ is not relevant as it will be canceled by $a_{j_{0}}=0$. In particular, if we define $\phi^u$ to be a solution generated by some $\binom{\phi^u(j_0+1)}{\phi^u(j_0)}\in E^u(j_{0}+1)$, then by Lemma ~\ref{l.expgrowth}, we will have for all $j\ge j_{0}+1$:
		\beq\label{eq:u_half}
		\left\|\binom{\phi^u(j+1)}{\phi^u(j)}\right\|=\left\|\binom{\phi^u(j_0+1)}{\phi^u(j_0)}\right\|\|A^E_{j-j_{0-1}}(j_{0}+1)\vec u(j) \| \ge \tilde k \lambda^{j-j_{0}}.
		\eeq
		It is clear that the above estimate implies that $|\phi^u(n_l)|>\tilde k\l^{n_l-j_{0}}$ for a strictly monotone increasing sequence $n_l>j_{0}$. In particular, it is a nontrivial solution which must form a basis of the solution space. Hence, all nontrivial solutions of $J_{(j_{0},+)}\phi=E\phi$ are some $\phi^u$ generated by a nonzero vector in $E^{u}(j_{0}+1)$ and it holds that for some $\tilde k=\tilde k(\phi^u)>0$,
		\beq\label{eq:u_half1}
		|\phi^u(n_l)|>\tilde k\l^{n_l-j_{0}}\mbox{ for all }l\ge 1.
		\eeq
		Going backwards in time $j$ and running the same proof above, we obtain the following information in case (2). Any nontrivial solution $\phi^s$ of $J_{(-,j_0]}\phi=E\phi$ is generated by a vector in $E^s(j_0)$.  By case (2) of Lemma~\ref{l.expgrowth} and the estimate above we obtain for some $n_t<j_0$, which is strictly monotone decreasing in $t\ge 1$, that
		\beq\label{eq:s_half}
		|\phi^s(n_t)|>\tilde k\l^{j_0-n_t}\mbox{ for all }t\ge 1.
		\eeq
		
		To consider case (3), we fix any $j_0$. Then all solutions are of the form $\psi=\alpha\phi^u+\beta\phi^s$ where $\phi^u$ is generated by a vector in $E^u(j_0)$ and $\phi^s$ is generated by a vector in $E^s(j_0)$. Since $a_j\neq 0$ for all $j\in\Z$, the same proof as in the first two cases yield that $\phi^u(n_l)$ grows exponentially fast as $n_l\to\infty$ and $\phi^s(n_t)$ grows exponentially fast as $n_t\to-\infty$ where $\{n_l\}$ and ${n_l}$ are some sequences. Hence, if $\psi$ is nontrivial, it must grow exponentially fast at least along one of the sequences $\{n_l\}$ and $\{n_t\}$.
		
		Now, we consider case (4), i.e. the finite restrictions$J_{(j_1, j_2]}$. We want to show
		$$
		\dim \{\phi^f(n), j_1< n \le j_2: (J_{(j_{1},j_{2}]} -E)\phi^f(n)=0 \}=0.
		$$
		If not, then $(J_{(j_{1},j_{2}]} -E)\phi^f=0$ for some nonzero vector $\phi^f$. Set $N=j_2-j_1$. Then we have $p_{N}(j_{1}+1,E)=\det(E-J_{(j_{1},j_{2}]})=0$ since $E\in\sigma(J_{(j_1,j_2]})$. Recall $a_{j_{1}}=a_{j_{2}}=0$. Thus, we have
		$$
		B_N^E(j_1+1) = \begin{pmatrix}
			p_{N}(j_{1}+1,E) & -\overline{a_{j_{1}}}p_{N-1}(j_{1}+2,E) \\
			a_{j_{2}}p_{N-1}(j_{1}+1,E) & -\overline{a_{j_{1}}}a_{j_{2}}p_{N-2}(j_{1}+2,E) 
		\end{pmatrix}
		$$
		is a zero matrix, which contradicts condition (4) of Definition~\ref{d.uhsequence}. 
		Hence, the solutions of $(J_{(j_{1},j_{2}]} -E)\phi^f=0$ are all trivial. 
	\end{proof}
	Let $\mathcal{E}_g(J_{a,b})$ be the set of generalized eigenvalues of $J_{a,b}$, i.e. all $E \subset \C$ that admit a nontrivial polynomially bounded solution of $J_{a,b}\psi=E\psi$. From the theorem of Sch'nol-Berezanskii \cite{berezanskii,Sch}, it is well-known that
	$$
	\overline{\mathcal{E}_g(J_{a,b})}= \sigma(J_{a,b}).
	$$ 
	In particular, it suffices to show for any $E$ such that $B^E\in\CD\CS$, all nontrivial solutions of $J_{a,b}\psi=E\psi$ are not polynomially bounded. Indeed, by Theorem~\ref{l.opennessDS}, $B^{E'}\in\CD\CS$ for all $E'$ close to $E$. Hence, for all $E'$ close $E$, all nontrivial solutions of $J_{a,b}\psi=E'\psi$ are not polynomially bounded which implies $E\in\overline{\CE_g(J_{a,b})}^\complement=\rho(J_{a,b})$. 
	
	Hence, to prove
	$$
	\{E:B^{E} \in \CD\CS\ \}\subset\rho(J_{a,b}),
	$$
	we only need to show all nontrivial solutions of $J_{a,b}\psi=E\psi$ are not polynomially bounded when $B^E\in\CD\CS$. To this end, we define 
	\beq\label{eq:j_min_max}
	j_{\max}:=\sup \{j\in\Z:a_j=0\} \mbox{ and }j_{\min}:=\inf\{j\in\Z: a_j=0\},
	\eeq
	if they exist. Note we allow $j_{\max}$ to be $\infty$ and $j_{\min}$ to be $-\infty$. 
	
	\begin{proof}[Proof of the first direction of Theorem~\ref{t.uniformhers}]
	Now we are ready to prove \eqref{eq:firstdirection}. We may divide the proof into the following cases.
	
	Case I. $a_j\neq 0$ for all $j\in\Z$. Then  case (3) of Lemma~\ref{l.trivial} directly implies that all nontrivial solutions of $J_{a,b}\psi=E\psi$ are not polynomially bounded.
	
	Case II. $-\infty<j_{\min}\le j_{\max}<\infty$. If $j_{\min}=j_{\max}$, then we can decompose $J_{a,b}$ as the following direct sum
	$$
		J_{a,b}=J_{(-,j_{\min}]}\oplus J_{(j_{\max},+)}.
	$$
	If $j_{\min}<j_{\max}$, then we have
	$$
	J_{a,b}=J_{(-,j_{\min}]}\oplus J_{(j_{\min}, j_{\max}]}\oplus J_{(j_{\max},+)}.
	$$
	By cases (1), (2), and (4) of Lemma~\ref{l.trivial}, in both cases, the solution space is spanned by
	$$
	\{(\phi^s,0,0,\ldots), (\ldots,0,0, \phi^u)\}
	$$
	where $\phi^s$ and $\phi^u$ are any nontrivial solutions of $J_{(-,j_{\min}]}\phi=E\phi$ and $J_{(j_{\max},+)}\phi=E\phi$, respectively. Hence, by \eqref{eq:u_half1} and \eqref{eq:s_half} and similar to Case I, all nontrivial solutions of the eigenvalue equation in both cases must grow exponentially fast at least along some subsequences go to $\infty$ or $-\infty$.
	
	Case III. Either $j_{\min}=-\infty$ and $j_{\max}<\infty$; or $j_{\min}>-\infty$ and $j_{\max}=\infty$. In the first case, there exists a strictly monotone decreasing sequence of integers $\{j_k\}_{k\ge 1}$ such that $j_1=j_{\max}$ and $a_{j_k}=0$ for all $k\ge 1$. Hence we may decompose the operator $J_{a,b}$ as 
	$$
	J_{a,b}=\ldots\oplus J_{(j_{k+1}, j_k]}\oplus\cdots \oplus J_{(j_{2}, j_1]}\oplus J_{(j_1,+)}.
	$$
	Then by cases (1) and (4) of Lemma~\ref{l.trivial}, the solution space of $J_{a,b}\phi=E\phi$ is spanned by $(\ldots, 0,0,\phi^u)$, where $\phi^u$ is any nontrivial solution of $J_{(j_{\max},+)}\phi=E\phi$. By \eqref{eq:u_half1}, all nontrivial solutions of $J_{a,b}\phi=E\phi$ grows exponentially fast along $\{n_l\}$ where $n_l\to\infty$ as $l\to\infty$.
	
	In the second case, we may decompose the operator $J_{a,b}$ as 
	$$
	J_{a,b}=J_{(-,j_1)}\oplus J_{(j_{1}, j_2]}\oplus \cdots\oplus J_{(j_{k}, j_{k+1}]}\oplus\cdots 
	$$
	where $j_1=j_{\min}$ and $a_{j_k}=0$ for all $k\ge 1$. Hence, the solution space of $J_{a,b}\phi=E\phi$ is spanned by $(\phi^s, 0, 0\ldots)$, where $\phi^s$ is any nontrivial solution of $J_{(-,j_{\min}]}\phi=E\phi$. By \eqref{eq:s_half}, all nontrivial solutions of $J_{a,b}\phi=E\phi$ grow exponentially along $\{n_t\}$ where $n_t\to-\infty$ as $t\to\infty$.
	
	Case IV. $j_{\min}=-\infty$ and $j_{\max}=\infty$. Then there exists a subsequence $(j_{k})_{k \in \Z}$ such that $j_k\to\pm \infty$ as $k\to\pm\infty$ and $a_{j_{k}}=0$ for all $k\in\Z$. Hence, we may decompose $J_{a,b}$ as 
	$$
	J_{a,b} = \bigoplus^{\infty} _{k=-\infty}J_{(j_{k},j_{k+1}]}.
	$$
	By case (4) of Lemma~\ref{l.trivial}, all solutions of $J_{a,b}\phi=E\phi$ are trivial.
	
	Since the four cases exhaust all the possibilities, the proof is completed.
	
	\end{proof}

	\section{Domination Away from the Spectrum}\label{ss:NotUHtoSpectrum} 
	
	In this section we prove the second half \eqref{eq:secondhalf} of Theorem~\ref{t.uniformhers}. That is,
	$$
	\rho(J_{a,b}) \subset \{E:B^{E} \in \CD\CS\ \}.
	$$
	
	\subsection{Lower bound of the norm of iterations of Jacobi cocycles}\label{ss:lowerbound_norm}
	In this section, we show if $E\in\rho(J_{a,b})$, then $B^E$ satifies the stronger version \ref{eq:equiv_c4} of condition (4) of Definition~\ref{d.uhsequence}. For simplicity, we denote by $J_N(j)=J_{[j,j+N)}$, the restriction of $J_{a,b}$ on $[j, N)$. Recall we have $p_N(j,E) = \det (E - J_N(j))$. Let $\mathrm{diam}(S)$ be the diameter of a subset $S\subset\C$.
	
	\begin{lemma}\label{l.Weyls}
		Let $E\in\rho(J_{a,b})$ and fix a $N\in\Z_+$. For all $\varepsilon >0$, there exists a $\delta =\delta(\e)>0$ such that the following is true: if $|p_N(j,E)|\le \delta$ for some $j\in\Z$, then there exists an $E_0\in\sigma(J_N(j))$ such that $|E-E_0|< \varepsilon$.
	\end{lemma}

	\begin{proof}
		To prove present lemma, it suffices to show the following result. 
		
		\emph{Fix $N\in\Z_+$ and define $I_N(j, \delta):= \{E \in \C: |p_N(j, E)|\le \delta\}$. For all $\e >0$, there exists $\delta =\delta(\e)>0$ such that for all $j \in \Z$ and all connected components $\CK$ of $I_N(j, \delta)$, we have
			$$
			\mathrm{diam}(\CK)<\e.
			$$
		Moreover, there is a zero of $p_N(j, E)$ in $\CK$.}
		
        To show that each connected component of $I_N(j,\delta)$ contains a zero of $p_N(j,E)$, we just need to show that $|p_N(j,E)|$ has no positive local minimum. Suppose $|p_N(j,E)|$ has a positive local minimum at $E_0$. Then $p_N(j,E): B_r(E_0)=\{E:|E-E_0|< r\}\to\C$ has no zeros for small $r>0$. Hence, the holomorphic function $\frac1{p_N(j,E)}:B_r(E_0)\to\C$ attains its maximum modulus  at an interior point $E_0$. This contradicts the maximum principle since $p_N(j,E)$ is nonconstant for all $j\in\Z$ and all $N\in\Z_+$.  
		
		Next we show for all $\e>0$ the existence of desired $\delta$. Here we use the following version of the so-called Markov's inequality, see. e.g. \cite[Lemma 3.1]{siciak}:
		
		For any connected compact set $\CK\subset\C$ with $\mathrm{diam}(\CK)>\eta>0$, there exist positive constants $M, \alpha$, depending only on $\eta$, such that it holds for all polynomials $p(z)$ and all $r\in\N$ that
		\beq\label{eq:markoff}
		\left\|\frac{d^rp_n}{dz^r}\right\|_{\CK} \le M(\deg p)^{r\alpha}\|p_n\|_{\CK}, 
		\eeq
		where $\|\cdot\|_{\CK}$ denotes the supremum norm on $\CK$ and $\deg p$ is the degree of the polynomial $p$. Now suppose the statement at the begining of this proof is not true. Then there exists $\e_0>0$ such that for all $\delta>0$, there exists $j \in \Z$ and a connected component $\CK$ of $I_N(j,\delta)$ whose diameter satisfies
		$$
		\mathrm{diam}(\CK)\ge \e_0.
		$$
		Applying \eqref{eq:markoff} to $p_N(j, E)$ on $\CK$, we obtain
		$$
		\left\|\frac{d^{N}p_N(j, E)}{dE^N}\right\|_{\CK}\le MN^{r\alpha}\|p_N(j,E)\|_{\CK} \le MN^{N\alpha}\delta.
		$$
		It is clear that the coefficient of the highest order term $E^N$ in $p_N$ is $1$. Hence, we obtain for all $\delta>0$ that
		$$
		N!\le MN^{N\alpha}\delta
		$$
		where $N$ is fixed and $M$ and $\a$ depend only on $\e_0$. This is clearly not possible if we choose $\delta$ to be smaller than $\frac{N!}{MN^{N\a}}$.
	\end{proof}
	
	\bigskip
	
	First, we note the following facts from functional analysis. Let $H$ be a self-adjoint operator, i.e. $H^*=H$, on a Hilbert space $\mathbb H$. Then Weyl's criterion (see e.g. \cite{reedsimon}) says:
	\beq\label{eq:weyl}
      E\in\sigma(H) \mbox{ if and only if : for all }\e>0,\ \|(H-E)u\|< \e \mbox{ for some unit vector }u\in \mathbb H.
	\eeq
	Moreover, for each $z\in\rho(H)$, we have
	\beq\label{eq:inverseOperatorNorm}
	\|(H-zI)^{-1}\|^{-1}=d(z,\sigma(H)).
	\eeq
	Clearly, $J_{a,b}$ is self-adjoint. We define $\delta_j(n)\in\ell^2(\Z)$ to be the sequence such that 
	$$
	\delta_j(n)=\begin{cases} 1,& \mbox{ if } m=j\\
		0, & \mbox{ if }m\neq j.
		\end{cases}
	$$ 
	
	  Throughout this section, we fix $E\in\rho(J_{a,b})$ and we let $M>0$ be that 
	  $$
	  \sup_{j\in\Z}\{|a_j|,|b_j|, |b_j-E|, |E|, \|J_{a,b}\|\}<M.
	  $$
	  Moreover, we define
	\beq\label{eq:dist_to_spectrum}
	\delta:=d(E,\sigma(J_{a,b}))=\|(J_{a,b}-E)^{-1}\|^{-1}
	\eeq
	\begin{lemma}\label{l.lowboundbn}
		Let $E\in\rho(J_{a,b})$. Then for all $n \in \Z_+$, it holds that $$\inf_{j \in \Z} \|B^E_n(j)\|>0.$$
	\end{lemma}

	\begin{proof}
		We will prove this lemma by induction. First, we set $n=1$. We need to show
		$$
		\inf_{j \in \Z} \|B^E(j)\|>0.
		$$
		Suppose $\inf_{j \in \Z} \|B^E(j)\|=0$. Then, for all $\e>0$, there exists $j_0 \in \Z$ such that $\|B^E(j_0)\|< \epsilon$. Hence, we have
		$$
		\left\| \begin{pmatrix} E-b_{j_0}& -\overline{a_{j_0-1}} \\\ a_{j_0} & 0\end{pmatrix} \right \| < \e   \mbox{ implies } |E-b_{j_0}|< \e,\ |a_{j_0-1}|< \e,\mbox{ and } |a_{j_0}|< \e.
	     $$
		Now, let's consider $(J_{a,b}-E)\delta_{j_0}.$ Clearly, 
		$$
		(J_{a,b}-E)\delta_{j_0}(n)=\overline{a_{n-1}}\delta_{j_0}(n-1)+a_{n}\delta_{j_0}(n+1)+(b_n-E)\delta_{j_0}(n)
		$$
		which is 0 for all $n \not = j_0-1, j_0, j_0+1$. On the other hand, it is easy to see that
		$$ 
		(J_{a,b}-E)\delta_{j_0}(n)= \begin{cases}
			a_{j_0-1}, &\mbox{ if } n=j_0-1\\
			b_{j_0}-E, &\mbox{ if } n=j_0 \\
			\overline{a_{j_0}}, &\mbox{ if } n=j_0+1.
		\end{cases}
	$$
		Hence, $\|(J_{a,b}-E)\delta_{j_0} \|_{\ell^2} < 3\e$. By Weyl's Criterion,  we have $E\in\sigma(J_{a,b})$ since $\|\delta_{j_0}\|_{\ell^2}=1$, which contradicts $E\in\rho(J_{a,b})$.
		
		Next, we do the induction: assuming 
		$$
		\inf_{1\le n<N}\inf_{j \in \Z} \|B^E_n(j)\|=c>0,
		$$
		we want to show $\inf_{j \in \Z} \|B^E_N(j)\|>0$. Recall we have $p_N(j,E) = \det (E - J_N(j))$ and 
		$$
		B_N^E(j) = \begin{pmatrix}
			p_N(j,E) & -\overline{a_{j-1}}p_{N-1}(j+1,E) \\
			a_{j+N-1}p_{N-1}(j,E) & -\overline{a_{j-1}}a_{j+N-1}p_{N-2}(j+1,E)
		\end{pmatrix}.
		$$
		By Lemma ~\ref{l.Weyls}, for all $\e>0$, there exists $\delta >0$ such that if $|p_N(j,E)|< \delta$, then there exists $E_0\in\sigma(J_N(j))$  such that $|E-E_0|<\e$. Assume $\inf_{j \in \Z} \|B^E_N(j)\|=0$. Choose $0<\gamma < \max\{\frac{\e^2}{2}, \delta, \frac{1}{M^2}\}$. Then, there exists $j_0 \in \Z$ such that $\|B^E_N(j_0)\|< \gamma^{3N}$ which implies all the following four terms
		$$
			|p_N(j_0,E)|,\ 
			|a_{j_0-1}p_{N-1}(j_0+1,E)|,\ 
			|a_{j_0+N-1}p_{N-1}(j_0,E)|,\mbox{ and }
			$$
			$$
			|a_{j_0-1}a_{j_0+N-1}p_{N-2}(j_0+1,E)|
		$$
		are smaller than $\gamma^{3N}$. We will divide the discussion into three different cases.
		
		Case I. $|a_{j_0+N-1}| < \gamma$, $|a_{j_0-1}| < \gamma$. Since $|P_N(j_0,E)|<
		\gamma^{3N}<\delta$, by Lemma ~\ref{l.Weyls}, there exists $E_0\in\sigma(J_{N}(j_0))$ such that $|E-E_0|< \e$. Choose a unit eigenvector $\vec u=(u_1,\ldots, u_N)$ of  $J_N(j_0)$ for the eigenvalue $E_0$. We define $\phi=(\ldots,0, 0, \vec u, 0, 0,\ldots)$ which is a unit vector in $\ell^{2}(\Z)$. Clearly,
		$(J_{a,b}-E)\phi=(J_{a,b}-E_0)\phi+(E-E_0)\phi$ where $\|(E-E_0)\phi\|<\e$. 
		A direct computation shows that
		$$
		[(J_{a,b}-E_0)\phi]_j=
		\begin{cases}
			0, &j_0\le j\le j_0+N-1,\ j\ge j+N+1, \mbox{ or }j\le j_0-2\\
			a_{j_0-1}u_1, &j=j_0-1\\
			\overline{a_{j_0+N-1}}u_N, &j=j+N,
		\end{cases}
		$$
		which implies $\|(J_{a,b}-E_0)\phi\|<2\gamma<2\e$. Hence, it holds for the unit vector $\phi$ that
	     $$
	     \|(J_{a,b}-E)\phi\|<3\e.
	     $$
		
		Case II. Either $|a_{j_0+N-1}| < \gamma$ and $|a_{j_0-1}|\ge\gamma$; or $|a_{j_0+N-1}|\ge \gamma$ and $|a_{j_0-1}|<\gamma$. Without loss of generality, we focus on the case $|a_{j_0+N-1}| < \gamma$ and $|a_{j_0-1}|\ge\gamma$. First, we have
		$$
		\gamma|p_{N-1}(j_0+1,E)|\le |a_{j_0-1}p_{N-1}(j_0+1,E)|<\gamma^{3N}
		$$
		which implies $|p_{N-1}(j_0+1,E)| < \gamma^{3N-1}$. If $|a_{j_0}|< \sqrt{2\gamma}<\e$, then together with
		$$
		|p_{N-1}(j_0+1,E)| < \gamma^{3N-1}<\delta \mbox{ and }|a_{j_0+N-1}|<\gamma<\e.
		$$ 
		We may apply the same argument as in case I to the operator $J_{N-1}(j_0+1)$ and get a unit vector $\phi\in\ell^2(\Z)$ such that
		$$
		\|(J_{a,b}-E)\phi\|<3\e.
		$$
		If  $|a_{j_0}|\ge \sqrt{2\gamma}$, then it is a standard result that
		$$
		|p_N(j_0,E)|=|(E-b_{j_0})p_{N-1}(j_0+1,E)-|a_{j_0}|^2p_{N-2}(j_0+2,E)| < \gamma^{3N}
		$$
		which implies
       $$
		|a_{j_0}^2p_{N-2}(j_0+2,E)| < \gamma^{3N-2} + \gamma^{3N} < 2\gamma^{3N-2},
       $$
       which in turn implies
		$$
		|p_{N-2}(j_0+2,E)| <\gamma^{3N-3}.
		$$ 
		Repeating this procedure, then either at some step we obtain $\|(J_{a,b}-E)\phi\|<3\e$ for some unit vector $\phi\in\ell^2(\Z)$, or we eventually get 
		$$
		\big|p_2(j_0+N-2,E)|=|(E-b_{j_0+N-2})(E-b_{j_0+N-1}) -|a_{j_0+N-2}|^2\big|<\gamma^9
		$$ 
		and 
		$$
		|p_1(j_0+N-1,E)|=|E-b_{j_0+N-1}| <\gamma^6.
		$$
		So we have that $|E-b_{j_0+N-1}| <\gamma^6<\e$, $|a_{j_0+N-1}| <\gamma<\e$ and $|a_{j_0+N-2}| <\sqrt{2\gamma^5} <\e$, which again yields
		$$
		\|(J_{a,b}-E)\phi\|<3\e.
		$$
		
		Case III. $|a_{j_0+N-1}| \ge \gamma$ and  $|a_{j_0-1}|  \ge \gamma$. In this case, we must have $|p_{N-1}(j_0+1,E)|<\gamma^{3N-1} $, $|p_{N-1}(j_0,E)| <\gamma^{3N-1}$ and $|p_{N-2}(j_0+1,E)| <\gamma^{3N-2}$. Note we have
		$$
		|p_N(j_0,E)| = |(E-b_{j_0})p_{N-1}(j_0+1,E)-|a_{j_0}|^2p_{N-2}(j_0+2,E)| <\gamma^{3N}\mbox{ and }
		$$
		$$
		|p_{N-1}(j_0,E)|= |(E-b_{j_0})p_{N-2}(j_0+1,E)-|a_{j_0}|^2p_{N-3}(j_0+2,E)| <\gamma^{3N-1}.
		$$
		So we must have that 
		$$
		|a_{j_0}^2p_{N-2}(j_0+2,E)| <2\gamma^{3N-2} \mbox{ and } |a_{j_0}^2p_{N-3}(j_0+2,E)| <2\gamma^{3N-3} 
		$$
		If $|a_{j_0}| >\sqrt{2\gamma}$, then $|p_{N-2}(j_0+2,E)| <\gamma^{3N-3} $ and $|p_{N-3}(j_0+2,E)| <\gamma^{3N-4}$. We have
		$$B^E_{N-1}(j_0+1) = \begin{pmatrix}
			p_{N-1}(j_0+1,E) & -\overline{a_{j_0}}p_{N-2}(j_0+2,E) \\
			a_{j_0+N-1}p_{N-2}(j_0+1,E) & -\overline{a_{j_0}}a_{j_0+N-1}p_{N-3}(j_0+2,E)
		\end{pmatrix}
		$$
		 Hence, combine all the estimates above, we obtain the following estimate
		$$
		\|B^E_{N-1}(j_0+1)\|<5\gamma^{3N-5}<5\gamma<5\e,
	   $$
	   which cannot happen if we choose $\e<\frac5c$. Thus we must have $|a_{j_0}|<\sqrt{2\gamma}<\e$. By the same argument as above, we see that $|a_{j_0+N-2}| <\sqrt{2\gamma}<\e$. Recall $|p_{N-2}(j_0+1,E)| <\gamma^{3N-2}<\delta$. Applying the same argument as in case I to $J_{N-2}(j_0+1)$, we again obtain
	   $$
	   \|(J_{a,b}-E)\phi\|<3\e
	   $$
	   for some unit vector $\phi\in\ell^2(\Z)$.
	   
	   Combining all the possible cases, if $\inf_{j\in\Z}B_N(j)=0$, then for all $\e>0$ we can find a unit vector $\phi\in\ell^2(\Z)$ such that
	   $$
	   \|(J_{a,b}-E)\phi\|<3\e,
	   $$
	   which by Weyl's criterion implies $E\in\sigma(J_{a,b})$. This contradicts $E\in\rho(J_{a,b})$, concluding the proof.
	\end{proof} 

\subsection{Estimates on and structure of the Green's Function}\label{ss:green}
	
	Now, to prove that $B^{E} \in \CD\CS$, we only need to construct the two invariant directions and show they satisfy conditions (1)-(3) of Definition~\ref{d.uhsequence}. First, we perform a Combes-Thomas type of estimate \cite{combesthomas} concerning exponential decay of the Green's Function. 
    
    \begin{lemma}\label{l.combes}
    For each $E\in\rho(J_{a,b})$. There exists a positive constant $\gamma=\gamma(\delta,M)>0$  such that 
    \beq\label{eq:green_decay}
    |(J_{a,b}-E)^{-1}(p,q)|\le \frac2\delta e^{-\gamma |p-q|}\mbox{ for all }p,q\in\Z.
    \eeq
    \end{lemma}

    \begin{proof} Define $M_\b$ to be the multiplication operator $(M_\b\phi)_n=e^{\b n}\phi_n$. Without loss of generality, we may assume $|\beta|\le 1$. Let $T$ be the left shift operator. That is, $(T\phi)_n=\phi_{n+1}$. A direct computation shows that:
	\begin{align*}
		(M_{-\b}(J_{a,b}-E)M_\b \phi)_n &=e^{-\b n} (J_{a,b}-E)(e^{\b n} \phi_n)\\
		&= e^{-\b n}(\overline{a_{n-1}}e^{\b(n-1)}\phi_{n-1}+a_{n}e^{\b(n+1)}\phi_{n+1}+(b_n-E)e^{\b n}\phi_{n}) \\
		&= \overline{a_{n-1}}e^{-\b}\phi_{n-1}+a_{n}e^{\b}\phi_{n+1}+(b_n-E)\phi_{n}\\
		&= (J_{a,b}-E)\phi_{n}+a_n(e^\b-1)\phi_{n+1}+\overline{a_{n-1}}(e^{-\b}-1)\phi_{n-1} \\
		&= (J_{a,b}-E)\phi_{n}+a_n(e^\b-1)(T\phi_{n})+\overline{a_{n-1}}(e^{-\b}-1)(T^{-1}\phi_{n}).
	\end{align*}
	Hence, 
	$$
	M_{-\b}(J_{a,b}-E)M_\b =J_{a,b}-E+a_n(e^\b-1)T+\overline{a_{n-1}}(e^{-\b}-1)T^{-1} =J_{a,b}-E+S.
	$$
	 The operator $S$ is bounded on $\ell^2(\Z)$ and it holds for some $C=C(M)>0$ that
	$$
	\|S\|\le |a_n||(e^\b-1)|+|a_{n-1}||(e^{-\b}-1)|\le C|\b|.
	$$
	Clearly, $\|(J_{a,b}-E)^{-1}S\|\le \frac12$ if $|\b|\le \|(J_{a,b}-E)^{-1}\|^{-1}/(2C)=\delta/(2C)$. Then
	$$
	M_{-\b}(J_{a,b}-E)M_\b=J_{a,b}-E+S=(J_{a,b}-E)[I+(J_{a,b}-E)^{-1}S]
	$$
	is invertible. Moreover 
	$$
	(M_{-\b}(J_{a,b}-E)M_\b)^{-1}=M_{-\b}(J_{a,b}-E)^{-1}M_\b=[I+(J_{a,b}-E)^{-1}S]^{-1}(J_{a,b}-E)^{-1},
	$$
	which implies
	$$
	\|M_{-\b}(J_{a,b}-E)^{-1}M_\b\|\le 2\|(J_{a,b}-E)^{-1}\|=\frac2\delta.
	$$
	Hence, it holds for all $p,q\in\Z$ that
	\begin{align*}
		|\langle\delta_p, M_{-\b}(J_{a,b}-E)^{-1}M_\b\delta_q\rangle|&=|\langle M_{-\b}\delta_p, (J_{a,b}-E)^{-1}M_\b\delta_q\rangle|\\
		&=|(J_{a,b}-E)^{-1}(p,q)|e^{-\b(p-q)}\\
		&\le \frac2\delta.
	\end{align*}
	By choosing the sign of $\beta$ appropriately, it clearly holds that for $\gamma=|\b|$ that
	$$
	|(J_{a,b}-E)^{-1}(p,q)|\le \frac2\delta e^{-\gamma|p-q|}
	$$
	which is nothing other than \eqref{eq:green_decay}.
	\end{proof}
	
	We define $g_j(n)=(J_{a,b}-E)^{-1}(n,j)$. By \eqref{eq:green_decay}, it holds that
	\beq\label{eq:decay_g_j}
	|g_j(n)|<\frac2\delta e^{-\gamma|n-j|},\mbox{ for all } n,j\in\Z.
	\eeq
	It is clear that $(g_j(n))_{n\in\Z}$ and $(g_{(j)}(n))_{j\in\Z}$ are the unique solutions of the equations
	\beq\label{eq:g}
	(J_{a,b}-E)g_j=\delta_j\mbox{ and }g_{(\cdot )}(n)(J_{a,b}-E)=\delta_n.
	\eeq
	In other words, we have for each $j\in\Z$ that
	\beq\label{eq:large_g_value1}
	\overline{a_{j-1}}g_{j}(j-1)+a_{j}g_{j}(j+1)+(b_{j}-E)g_{j}(j)=1 \mbox{ and }
	\eeq
	\beq\label{eq:large_g_value2}
	\overline{a_{j}}g_{j+1}(j) +a_{j-1}g_{j-1}(j)+(b_{j}-E)g_{j}(j)=1.
	\eeq
	In particular, we have for all $j\in\Z$ that
	\beq\label{eq:nonzero_g}
	|g_{j}(j-1)|+|g_{j}(j+1)|+|g_{j}(j)|\neq 0\mbox{ and }|g_{j-1}(j)|+|g_{j}(j)|+|g_{j+1}(j)|\neq 0.
	\eeq

	For each $j$, $g_j$ may be obtained by connecting pieces which are related to different solutions of $(J_{a,b}-E)\psi=0$. Such information will be the key to construct the two invariant directions of $B^E$. To get more such information about $g_j$, we divide the discussion into several different cases depending on whether $a_j=0$ or not. We denote by $\ell^2(\Z_\pm)$ all the two-sided infinite sequences who is square summable on $\Z_\pm$, respectively. 
	
	\textbf{Case I.} $a_n \not =0$ for all $n \in \Z$. Fix a $j\in\Z$. By \eqref{eq:nonzero_g}, $|g_{j}(j-1)|+|g_{j}(j+1)|+|g_{j}(j)|\neq 0$. We further subdivide it into three different cases. 
	
\textit{Case I.a.} $g_j(j)\neq 0$ or $g_j(j-1)g_j(j+1)\neq 0$. If we use $\binom{g_j(j)}{g_j(j-1)}$ as an initial condition, then we obtain get a nonzero solution  $\phi^u(n)$ of $J_{a,b} \phi=E \phi$ that coincides with $g_j(n)$ for $n\le  j$. Hence, $\phi^u(n) \in \ell^2(\Z_-)$. If we use $\binom{g_{j}(j+1)}{g_{j}(j)}$ as an initial condition, then we obtain a nonzero solution of  $\phi^s(n)$ of $J_{a,b} \phi=E \phi$ that coincides with $g_j(n)$ for $n\ge j$. Hence, $\phi^s(n) \in \ell^2(\Z_+)$. 
	
\textit{Case I.b.} $g_j(j-1)\neq 0$ and $g_j(j)=g_j(j+1)=0$. First, we have a nonzero solution $\phi^u_j\in\ell^2(\Z_-)$ that is generated by $\binom{g_j(j)}{g_j(j-1)}$. We claim that in this case we must have 
$$
|g_{j-1}(j)|+|g_{j-1}(j+1)|\neq0.
$$
Indeed, if $g_{j-1}(j)=g_{j-1}(j+1)=0$. Then $g_{j}(j)=g_{j-1}(j)=0$ and \eqref{eq:large_g_value2} together imply $g_{j+1}(j)\neq 0$. Hence $\binom{g_{j+1}(j)}{g_{j+1}(j-1)}$ is linearly independent with $\binom{g_j(j)}{g_j(j-1)}=\binom{0}{g_j(j-1)}$ which can generate another $\phi^u_{j+1}\in\ell^2(\Z_-)$. It is clear that  $\phi^u_j$ and $\phi^u_{j+1}$ are linear independent. This is impossible since 
	 \begin{align}
	\nonumber  \left|\det \begin{pmatrix}\phi^u_j(j)  & \phi^u_{j+1}(j)\\ \phi^u_j(j-1)  & \phi^u_{j+1}(j-1)\end{pmatrix}\right|&=\left|\det \left(A_{j-n}(n)\begin{pmatrix}\phi^u_j(j-n)  & \phi^u_{j+1}(j-n)\\ \phi^u_j(j-n-1)  & \phi^u_{j+1}(j-n-1)\end{pmatrix}\right)\right|\\
	 \nonumber &=|\det A_{n-j}(n)|\cdot\left|\det\begin{pmatrix}\phi^u_j(j-n)  & \phi^u_{j+1}(j-n)\\ \phi^u_j(j-n-1)  & \phi^u_{j+1}(j-n-1)\end{pmatrix}\right|\\
	 \label{eq:1dim_sol_space}&=\frac{|a_{j-n}|}{|a_{j-1}|}\left|\det\begin{pmatrix}\phi^u_j(j-n)  & \phi^u_{j+1}(j-n)\\ \phi^u_j(j-n-1)  & \phi^u_{j+1}(j-n-1)\end{pmatrix}\right|\\
	\nonumber  &\to 0\mbox{ as }n\to \infty
	 \end{align}
    while the first term is a fixed positive number, a contradiction. Thus, we may use $\binom{g_{j-1}(j+1)}{g_{j-1}(j)}$ to generate a nonzero solution $\phi^s$ of $J_{a,b}\phi=E\phi$ that coincides with $g_{j-1}(n)$ for all $n\ge j$. Hence, we must have $\phi^s\in \ell^2(\Z_+)$. 
  
  \textit{Case I.c.} $g_j(j+1)\neq 0$ and $g_j(j)=g_j(j-1)=0$. It is basically the dual case of case 1.b. First, we may construct a $\phi^s_j\in\ell^2(\Z_+)$, solution of $J_{a,b}\phi=E\phi$, via $\binom{g_j(j+1)}{g_j(j)}$. Then we must have
  $$
  |g_{j+1}(j)|+|g_{j+1}(j-1)|\neq0.
  $$
   Otherwise, $g_{j+1}(j)=g_j(j)=0$ and \eqref{eq:large_g_value2} together imply $g_{j-1}(j)\neq 0$. Thus $\binom{g_{j-1}(j+1)}{g_{j-1}(j)}$ can be used to generate another solution $\phi^s_{j-1}\in\ell^2(\Z_+)$ of $J_{a,b}\phi=E\phi$ that is linearly independent of $\phi^s_{j}$. This is again impossible since
   
  \begin{align}
  	\nonumber  \left|\frac{a_{j+1}}{a_n}\right|\cdot \left|\det \begin{pmatrix}\phi^s_{j-1}(j+1)  & \phi^s_{j}(j+1)&\\ \phi^s_{j-1}(j)  & \phi^s_{j}(j)\end{pmatrix}\right|&=\left|\det \left(A_{n-j-1}(j+1)\begin{pmatrix}\phi^s_{j-1}(j+1)  & \phi^s_{j}(j+1)&\\ \phi^s_{j-1}(j)  & \phi^s_{j}(j)\end{pmatrix}\right)\right|\\
    \label{eq:1dim_sol_space2}
  	&=\left|\det \begin{pmatrix}\phi^s_{j-1}(n)  & \phi^s_{j}(n)\\ \phi^s_{j-1}(n-1)  & \phi^s_{j}(n-1)\end{pmatrix}\right|\\
  	\nonumber  &\to 0\mbox{ as }n\to \infty
  \end{align}
while the first term in this chain of equations can be bounded below by a fixed positive number. Thus we may use $\binom{g_{j+1}(j)}{g_{j+1}(j-1)}$ to generate a nonzero solution $\phi^u\in\ell^2(\Z_-)$ of $J_{a,b}\phi=E\phi$.
    
    To sum up, in case I, we always obtain two nonzero solutions $\phi^s\in\ell^2(\Z_+)$ and $\phi^u\in\ell^2(\Z_-)$ of $J_{a,b}\phi=E\phi$. It is clear that $\phi^s$ and $\phi^u$ must be linear independent. Otherwise, they both decay exponentially as $n\to\pm\infty$ which implies they are eigenvectors which constradicts $E\in\rho(J_{a,b})$. Since $\{\phi: J_{a,b}\phi=E\phi\}$ is a two dimensional space, we have for all $j\in\Z$
    \beq\label{eq:gj_and_phi_1}
    g_j(n)= \begin{cases}\phi_j^s(n), & n\ge j,\\
    	\phi_j^u(n), & n\le j\end{cases}
    \eeq
	where $\phi^{s(u)}_j\in\mathrm{span}\{\phi^{s(u)}\}$, respectively.

	\textbf{Case II}. There is a $j_0$ so that $a_{j_0}= 0$ and $a_n\neq 0$ for all $n<j_0$. We consider the operator $J_{(-,j_0]}$. We may fix any $j<j_0$. We again divide the discussion into three different cases.
	
	\textit{Case II.a.} $g_j(j)\neq 0$ or $g_j(j-1)g_j(j+1)\neq 0$, then similar to case I.a, we may first use the vector $\binom{g_j(j)}{g_j(j-1)}$ to generate a vector $(\phi^-(j))_{j\le j_0}$ such that
	\beq \label{list}
		[(J_{(-,j_0]}-E)\phi^-](j)=0 \mbox{ for all }j<j_0\mbox{ and } \phi^{-}\in\ell^2(\Z_-).
	\eeq
Note by the same computation as in \eqref{eq:1dim_sol_space}, the set of all such $\phi^-$ is a one-dimension space. Moreover, $\phi^-$ cannot be a solution of $J_{(-,j_0]}\phi=E\phi$ or else we can find a $\ell^2$-solution of $J_{a,b}\phi=E\phi$ via $\phi=(\phi_-,0,0,\ldots,)$. In other words, we must have
 \beq\label{eq:nonzero_boundary1}
\overline{a_{j_0-1}} \phi^-(j_0-1) + (E-b_{j_0})\phi^-(j_0) \not =0.
\eeq
Next, we can use $\binom{g_j(j+1)}{g_j(j)}$ to generate a solution $(\phi^s(j))_{j\le j_0}$ of $J_{(-,j_0]}\phi=E\phi$. Note the set of such vectors form a one-dimensional space as well since they are uniquely determined by $\phi^s(j_0)$. Note that it must hold $\phi^s\notin \ell^2((-\infty,j_0])$. Otherwise, we can again construct a nontrivial $\ell^2$-solution of $J_{a,b}\phi=E\phi$.

\textit{Case II.b.} $g_j(j)=g_j(j+1)=0$ and $g_j(j-1)\neq 0$ , then we can again use the similar argument of case I.b. First, we can first obtain a nontrivial $\phi^-\in \ell^2(\Z_-)$ via $\binom{g_j(j)}{g_j(j-1)}$ which satisfies \eqref{list}. Then the same argument of case I.b would imply $\binom{g_{j-1}(j+1)}{g_{j-1}(j)}\neq \vec 0$ which can be used to generate a nontrivial $\phi^s$ solving $J_{(-\infty,j_0]}\phi=E\phi$.

\textit{Case II.c.} $g_j(j)=g_j(j-1)=0$ and $g_j(j+1)\neq 0$, then we can first use $\binom{g_j(j+1)}{g_j(j)}$ to generate a solution $\phi^s$ of $J_{(-\infty,j_0]}\phi=E\phi$. Then we can show that 
$$
|g_{j+1}(j)|+|g_{j+1}(j-1)|\neq 0.
$$
Indeed, otherwise we have $g_{j+1}(j)=g_j(j)=0$ which together with \eqref{eq:large_g_value2} imply $g_{j-1}(j)\neq 0$. Then $\binom{g_{j-1}(j+1)}{g_{j-1}(j)}$ can be used to generate another solution of $J_{(-\infty,j_0]}\phi=E\phi$ which is linearly independent with $\phi^s$. This is impossible since the solution space of $J_{(-\infty,j_0]}\phi=E\phi$ is one-dimensional. Hence $\binom{g_{j+1}(j)}{g_{j+1}(j-1)}$ can be used to generate a nonzero $\phi^-$ as described in \eqref{list}.

To summarize, in case 2, we always obtain a nonzero $\phi^s$ solving $J_{(-,j_0]}\phi^s=E\phi^s$ and a nonzero $\phi^-$ as described in \eqref{list}.  Now for each $j\in\Z$, we consider $g_j(n)$ for $n\le j_0$. If $j>j_0$, then we must have 
\beq\label{eq:gj_and_phi_2}
g_j(n)=0 \mbox{ for all } n\le j_0.
\eeq
Otherwise, $((g_j(n))_{n\le j_0}, 0,0,\ldots,)$ will be nonzero $\ell^2$-solution of $J_{a,b}\phi=E\phi$. If $j\le j_0$, then we have 
	\beq\label{eq:gj_and_phi_3}
g_{j}(n)= \begin{cases}\phi^-_j(n), & n\le j,\\
	\phi_{j}^s(n), & j\le n\le {j_0}.
\end{cases}
\eeq
where $\phi_j^s\in\mathrm{span}\{\phi^s\}$ and $\phi_j^-\in\mathrm{span}\{\phi^-\}$.

\textbf{Case III.} There is a $j_0$ such that $a_{j_0}=0$ and $a_n\neq 0$ for all $n>j_0$. We consider the operator $J_{(j_0,+)}$. We fix a $j>j_0+1$. The discussion is completely analagou to case II. Roughly speaking, depending on which of $\{g_j(j-1), g_j(j), g_j(j+1)\}$ is nonzero, we can have three different cases:
\begin{enumerate}
	\item[a.] $g_j(j)\neq 0$ or  $g_j(j-1) g_j(j+1)\neq 0$;
	\item[b.] $g_j(j)=g_{j}(j+1)=0$ and $g_j(j-1)\neq 0$;
	\item[c.] $g_j(j)=g_{j}(j-1)=0$ and $g_j(j+1)\neq 0$.
\end{enumerate}
In all these cases, we can obtain a nonzero solution $(\phi^u(j))_{j>j_0}$ which solves $J_{(j_0,\infty)}\phi=E\phi$ via $\binom{g_p(j)}{g_p(j-1)}$ where $p=j$ or $j+1$. Such solution form a one-dimensional space as they can be uniquely determined by $\phi^s(j_0+1)$.  Moreover, we can obtain a nonzero $(\phi^+(j))_{j>j_0}$ via $\binom{g_q(j+1)}{g_q(j)}$ where $q=j$ or $j-1$ such that
\beq \label{list2}
[(J_{(j_0,+)}-E)\phi^+](j)=0 \mbox{ for all }j>j_0+1\mbox{ and } \phi^{+}\in\ell^2(\Z_+).
\eeq
By the computation of \eqref{eq:1dim_sol_space2}, we obtain the set of all such $\phi^+$ form a one-dimensional space. Moreover, $\phi^+$ cannot be a solution of $J_{(j_0,+)}\phi=E\phi$. Otherwise, we can construct a $\ell^2$-solution $\phi=(\ldots, 0, 0,(\phi^{+}(j))_{n>j_0}$ of $J_{a,b}\phi=E\phi$. That is, we must have 
\beq\label{eq:nonzero_boundary2}
a_{j_0+2}\phi^+(j_0+2)+(b_{j_0+1}-E)\phi^+(j_0+1)\neq 0.
\eeq 
Then for each $j\in\Z$, we consider $g_j(n)$ for $n>j_0$. If $j\le j_0$, we must have
\beq\label{eq:gj_and_phi_4}
g_j(n)=0 \mbox{ for all } n>j_0.
\eeq
If $j>j_0$, then we have 
\beq\label{eq:gj_and_phi_5}
g_{j}(n)= \begin{cases}\phi^+_j(n), & n\ge j,\\
	\phi_{j}^u(n), & j_0<n\le j.
\end{cases}
\eeq

\textbf{Case IV.} There is a $j_1<j_2$ such that $a_{j_1}=a_{j_2}=0$ and $a_n\neq 0$ for all $j_1<n<j_2$ if such $n$ exists. Then we consider the operator $J_{(j_1,j_2]}$. It is clear that in this case, we have 
$$
(J_{a,b}-E)^{-1}=(J_{(-,j_1]}-E)^{-1}\oplus (J_{(j_1,j_2]}-E)^{-1}\oplus (J_{(j_2,+)}-E)^{-1}
$$
 where $(J_{(j_1,j_2]}-E)^{-1}$ is a finite matrix of order $j_2-j_1$. Then we have the following subcases.
 
 \textit{Case IV.a.} $j_2-j_1=1$, then it clearly holds that 
 $$
 (J_{(j_1,j_2]}-E)^{-1}=\frac1{b_{j_2}-E}.
 $$ 

 \textit{Case IV.b.} $j_2=j_1+2$, then it is clear that $(J_{(j_1,j_2]}-E)^{-1}$ is a $2\times 2$ bounded, invertible matrix which may be written as 
$$
\begin{pmatrix} g_{j_1+1}(j_1+1) & g_{j_1+2}(j_1+1)\\ g_{j_1+1}(j_1+2)& g_{j_1+2}(j_1+2)\end{pmatrix}.
$$

It is clear what are $g_j(n)$ for $j_1<n\le j_2$ and for all $j$ in the two cases above.

 \textit{Case IV.c} If $j_2>j_1+2$, then we want to explore a bit more the structure of $(J_{(j_1,j_2]}-E)^{-1}$. It is clear that if we start with 
$$
a_{j_1+1}\phi(j_1+2)+(b_{j_1+1}-E)\phi(j_1+1)=0,
$$
then we can find a nonzero vector $(\phi^{(1)}(j))_{j\in(j_1,j_2]}$ such that $[(J_{(j_1,j_2]}-E)\phi^{(1)}](j)=0 $ for all $j_1<j<j_2$ and
\beq\label{eq:nonzero_boundary3}
 \overline{a_{j_2-1}}\phi^{(1)}(j_2-1)+(b_{j_2}-E)\phi^{(1)}(j_2)\neq 0.
\eeq
The set of such $\phi^{(1)}$ is clearly a one-dimensional space. Similarly, if we start with 
$$
\overline{a_{j_2-1}}\phi(j_2-1)+(b_{j_2}-E)\phi(j_2)=0,
$$
then we can obtain a nonzero vector $(\phi^{(2)}(j))_{j\in(j_1,j_2]}$ such that
$[(J_{(j_1,j_2]}-E)\phi^{(2)}](j)=0 $ for all $j_1+1<j\le j_2$ and
\beq\label{eq:nonzero_boundary4}
 a_{j_1+1}\phi^{(2)}(j_1+2)+(b_{j_1+1}-E)\phi^{(2)}(j_1+1)\neq 0.
\eeq
The set of such $\phi^{(2)}$ is a one-dimensional space. Thus we obtain the following information of $g_j$. If $j\le j_1$ and $j>j_2$, then it is clear that
$$
g_j(n)=0\mbox{ for all } j_1<n\le j_2.
$$
If $j_1<j\le j_2$, then it holds that 
\beq\label{eq:gj_and_phi_6}
g_{j}(n)= \begin{cases}\phi^{(2)}_j(n), & j\le n\le j_2,\\
	\phi_{j}^{(1)}(n), & j_1<n\le j.
\end{cases}
\eeq
where $\phi^{(i)}_j\in\mathrm{span}\{\phi^{(i)}\}$ for $i=1,2$. 

It is clear that cases I-IV together describe $g_j(n)$ for all $j, n\in\Z$. 

\subsection{Construction of the dominated splitting}\label{ss:construction_inv}

We begin this section by the following estimates which is the consequence of the classfications in Section~\ref{ss:green}. 

	\begin{lemma}\label{l:large_su_norm}
		There exists a $C=C(\delta,M)>0$ such that for all $j\in\Z$, it holds that
		\beq\label{eq:large_s_norm}
		\max\big\{|g_{j-1}(j)|,\ |g_{j-1}(j-1)|,\ |g_{j-2}(j)|,\ |g_{j-2}(j-1)|\big\}>C^{-1} \mbox{ and }
		\eeq
			\beq\label{eq:large_u_norm}
		\max\left\{|g_{j}(j)|,\ |g_{j}(j-1)|,\ |g_{j+1}(j)|,\ |g_{j+1}(j-1)|\right\}>C^{-1}.
		\eeq
		\end{lemma}
	\begin{proof}
We focus on the proof of \eqref{eq:large_s_norm} as the proof of \eqref{eq:large_u_norm} is completely analogous. By \eqref{eq:large_g_value1} and \eqref{eq:large_g_value2}, we have for all $j\in\Z$ that
\beq\label{eq:large_g_value3}
	\overline{a_{j-2}}g_{j-1}(j-2)+a_{j-1}g_{j-1}(j)+(b_{j-1}-E)g_{j-1}(j-1)=1\mbox{ and }
\eeq
\beq\label{eq:large_g_value4}
	\overline{a_{j-1}}g_{j}(j-1) +a_{j-2}g_{j-2}(j-1)+(b_{j-1}-E)g_{j-1}(j-1)=1.
\eeq
	Suppose \eqref{eq:large_s_norm} is not true. Then we may choose $\e>0$ small and $j\in\Z$ so that
	$$
	\max\big\{|g_{j-1}(j)|,\ |g_{j-1}(j-1)|,\ |g_{j-2}(j)|,\ |g_{j-2}(j-1)|\big\}<\e,
	$$
which together with \eqref{eq:large_g_value3} and \eqref{eq:large_g_value4} implies
	$$
	|\overline{a_{j-2}}g_{j-1}(j-2)|>\frac12\mbox{ and } |\overline{a_{j-1}}g_{j}(j-1)|>\frac12.
	$$
	In particular, there is a $C=C(\delta,M)>0$ so that  
	$$
	\min\big\{|a_{j-2}|,\ |a_{j-1}|,\ |g_{j-1}(j-2)|,\ |g_{j}(j-1)|\big\}>C^{-1}.
	$$ 
	If we choose $\e>0$ small enough, then
	\begin{align*}
	\left|\det\begin{pmatrix}g_{j-1}(j-1) & g_j(j-1)\\ g_{j-1}(j-2) & g_j(j-2)\end{pmatrix}\right|&=|g_{j-1}(j-1)g_{j}(j-2)- g_j(j-1)g_{j-1}(j-2)|\\
	&\ge |g_j(j-1)g_{j-1}(j-2)|-|g_{j-1}(j-1)g_{j}(j-2)|\\
	&\ge C^{-2}-\e \frac2\delta\\
	&>0.
	\end{align*}
	Since $a_{j-1}a_{j-2}\neq 0$, we may use $\binom{g_{p}(j-1)}{g_{p}(j-2)}$ as an initial condition to generate a $\phi_{p}$ for $p=j$ or $j-1$ where $\phi_j$ and $\phi_{j-1}$ are linearly independent by the estimate above.
	
	On the other hand, if $a_n\neq 0$ for all $n<j-1$, then we are in case I or II of Section~\ref{ss:green}. Thus, we have either $\phi_p=\phi^u_{p}$ in case I or $\phi_p=\phi^-_{p}$ in case II. If there is a $j_1<j-2$ such that $a_{j_1}=0$ and $a_{n}\neq 0$ for all $j_1\le n\le j-1$, then we are in case III or IV.c of Section~\ref{ss:green}. Thus, we have either $\phi_p=\phi^u_{p}$ in case III or $\phi_p=\phi^{(1)}_p$ in case IV.c. However, in all these cases $\phi_j$ and $\phi_{j-1}$ belong to a same one-dimensional space which implies they must be linear dependent, a contradiction. 
 	\end{proof}
 
We define for each $j\in\Z$ a pair of vectors
	$$
	\vec s(j):=
	\begin{cases}\dbinom{g_{j-1}(j)}{g_{j-1}(j-1)} &\mbox{ if }|g_{j-1}(j)|+|g_{j-1}(j-1)|>\frac{C^{-1}}2;\\
		\dbinom{g_{j-2}(j)}{g_{j-2}(j-1)}  &\mbox{ otherwise}
		\end{cases}
	$$
and
	$$
	\vec u(j):=
	\begin{cases}
		\dbinom{g_{j}(j)}{g_{j}(j-1)} &\mbox{ if } |g_{j}(j)|+|g_{j}(j-1)|>\frac {C^{-1}}2;\\
		\dbinom{g_{j+1}(j)}{g_{j+1}(j-1)} &\mbox{ otherwise. }
		\end{cases}
	$$
By \eqref{eq:large_s_norm} and \eqref{eq:large_u_norm} and by choosing $C=C(M,\delta)>0$ appropriately, we have
	\beq\label{eq:large_inv_vec}
	C^{-1}<\|\vec s(j)\|<C \mbox{ and } C^{-1}< \|\vec u(j)\|<C\mbox{ for all }j\in\Z.
	\eeq
	Note by \eqref{eq:large_g_value3}, it holds that
	\beq\label{eq:s_choose_j-1}
	\vec s(j)=\binom{g_{j-1}(j)}{g_{j-1}(j-1)} \mbox{ if } a_{j-2}=0.
	\eeq 
	Similarly, by $\overline{a_{j-1}}g_j(j-1)+a_jg_j(j+1)+(b_j-E)g_j(j)=1$, we must have
	\beq\label{eq:u_choose_j}
	\vec u(j)=\binom{g_{j}(j)}{g_{j}(j-1)} \mbox{ if } a_{j}=0.
	\eeq 
	
Thus for each $j\in\Z$, we may define two one-dimensional subspaces of $\C^2$ as 
$$
E^s(j)=\mathrm{span}\{\vec s(j)\} \in E^s(j) \mbox{ and }E^u(j)=\mathrm{span}\{\vec u(j)\}.
$$

Since cases I-IV of Section~\ref{ss:green} describe $g_j(n)$ for all $j, n\in\Z$, we can in pariticular have the following corollary which will be heavily used in the remain part of this section. 

\begin{corollary}\label{c.classificaiton_on_g}
	Fix a $j\in\Z$ so that $a_{j-1}a_j\neq 0$. We define
	$$
	j_1=\max\{n<j-1: a_n=0\}\mbox{ and } j_2=\min\{n\ge j:a_n=0\}.
	$$
Then, it holds that $g_p(n)=\phi^\a_p(n)$ for all $p\le n< j_2+1$ and $g_q(n)=\phi^\b_q(n)$ for all $j_1<n\le q$ where $q=j-1$ or $j-2$, $q=j$ or $j+1$, and the choices of $\a$ and $\b$ are as follows:
\begin{enumerate}
	\item $j_1=-\infty$ and $j_2=\infty$: case I of Section~\ref{ss:green} which implies $\a=s$ and $\b=u$. 
	\item $j_1=-\infty$ and $j_2<\infty$: case II of Section~\ref{ss:green} which implies $\a=s$ and $\b=-$. 
	\item $j_1>-\infty$ and $j_2=\infty$: case III of Section~\ref{ss:green} which implies $\a=+$ and $\b=u$. 
	\item  $j_1>-\infty$ and $j_2<\infty$: case IV.c of Section~\ref{ss:green} which implies $\a=(2)$ and $\b=(1)$. 
	\end{enumerate}
Moreover, if $j_1=j-2$, then by \eqref{eq:s_choose_j-1} we have $p=j-1$. So it still holds that $g_p(n)=\phi^\a_p(n)$ for all $p\le n< j_2+1$. Similarly, if $j_2=j$, then \eqref{eq:u_choose_j} implies $q=j$ and $g_q(n)=\phi^\b_q(n)$ for all $j_1<n\le q$ .
\end{corollary}
	Now, we first show the invariance of $E^s$ and $E^u$.
	\begin{lemma}\label{l.Binvariance}
		$E^s$ and $E^u$ are $B^E$--invariant. In other words, for all $j\in\Z$, it holds that
		$$
		B^E(j)\cdot E^u(j) = E^u(j+1) \mbox{ and } B^E(j)\cdot E^s(j) \subseteq E^s(j+1).
		$$
	\end{lemma}
	\begin{proof}
		Fix a $j\in\Z$. First, we show $E^s(j)\subseteq E^s(j+1)$. It suffices to show $B^E(j)\vec s(j)\in E^s(j+1)$. We divide the proof into the following three cases.
		
		\textit{Case 1.s}. Assume $a_{j-1}a_j \neq0$.  By Corollary~\ref{c.classificaiton_on_g}, it holds that
		$$
		B^E(j)\vec s(j)=\binom{\phi_p^\a(j)}{\phi_p^\a(j-1)}=a_j\binom{\phi_p^\a(j+1)}{\phi^\a_p(j)}\in E^s(j+1)
		$$

		\textit{Case 2.s}. Assume $a_{j-1}=0$. Then $J_{a,b}=J_{(-,j-1]}\oplus J_{[j,+)}$. Thus we always have $g_p(j)=0$ for $p=j$ or $j-1$. In other words, we have
		$$
		B^E(j)\vec s(j)=\begin{pmatrix}E-b_j & 0\\ a_j & 0\end{pmatrix}\binom{0}{g_p(j-1)}=\binom{0}{0}\in E^s(j+1).
		$$
		
	    \textit{Case 3.s}. Assume $a_{j}=0$. Then it always true that $\overline{a_{j-1}}g_p(j-1)+(E-b_j)g_p(j)=0$ for $p=j-1$ or $j-2$. Thus, we have 
	    $$
	    B^E(j)\vec s(j)=\begin{pmatrix}E-b_j & -\overline{a_{j-1}}\\ 0& 0\end{pmatrix}\binom{g_p(j)}{g_p(j-1)}=\binom{0}{0}\in E^s(j+1).
	    $$
		
Next, we show $B^E(j)E^u(j)=E^u(j+1)$. Clearly, it suffices to show $\binom{0}{0}\neq B^E(j)\vec u(j)\in E^u(j+1)$. We again divide the proof into three different cases.

        \textit{Case 1.u}. Assume $a_{j-1}a_{j}\neq 0$. Then by Corollary~\ref{c.classificaiton_on_g}, it holds that
        $$
        B^E(j)\vec u(j)=\binom{\phi_q^\b(j)}{\phi_q^\b(j-1)}=a_j\binom{\phi_q^\b(j+1)}{\phi^\b_q(j)}\in E^u(j+1).
        $$
	  It is clear that $B^E(j)\vec u(j)\neq \binom{0}{0}$ since $\det(B^E(j))\neq0$ and $\vec u(j)\neq\binom{0}{0}$.
	  
	  \textit{Case 2.u}. Assume $a_{j}=0$. Then by \eqref{eq:u_choose_j}, we must have $\vec u(j)=\binom{g_j(j)}{g_j(j-1)}$. 
	  $$
	  B^E(j)\vec u(j)=\begin{pmatrix}E-b_j & -\overline{a_{j-1}}\\ 0 & 0\end{pmatrix}\binom{g_j(j)}{g_j(j-1)}=\binom{-\overline{a_{j-1}}g_j(j-1)+(E-b_j)g_j(j)}{0}=\binom{1}{0}.
	  $$
	 On the other hand,$J_{a,b}=J_{(-,j]}\oplus J_{(j,+)}$ implies $g_m(j)=0$ for $m>j$. Hence, we have 
	 $$
	\vec u(j+1)= \binom{g_{q}(j+1)}{g_q(j)}=\binom{g_{q}(j+1)}{0}\neq\vec 0
	 $$
	  where $q=j+1$ or $j+2$. Thus, it holds that $\vec 0\neq B^E(j)\vec u(j)\in E^u(j+1)$.
	  
	  \textit{Case 3.u}. Assume $a_{j-1}=0$ and $a_j\neq 0$. First, $J_{a,b}=J_{(-,j-1]}\oplus J_{[j,+)}$ implies $g_m(j-1)=0$ for $m\ge j$. Since $a_j\neq 0$, we must have
	  $$
	  B^E(j)\vec u(j)=\begin{pmatrix}E-b_j & 0\\ a_j & 0\end{pmatrix}\binom{g_q(j)}{0}=\binom{(E-b_j)g_q(j)}{a_jg_q(j)}=g_p(j)\binom{E-b_j}{a_j}\neq 0.
	  $$
	  On the other hand, we have $\vec u(j+1)=\binom{g_m(j+1)}{g_m(j)}$ for $m=j+1$ or $j+2$ which must satisfy $[(J_{a,b}-E)g_m](j)=0$, i.e.
	  $$
     a_jg_m(j+1)+(b_j-E)g_m(j)=0.
	  $$
	  Hence $\vec u(j+1)=\binom{g_m(j+1)}{g_m(j)}\neq \vec 0$ is linearly dependent with $\binom{E-b_j}{a_j}$ as well. Thus, we again have $B^E(j)\vec u(j)\in E^u(j+1)$.
	
	 In all the four different cases, we always obtain
	   $$
	   B^E(j)E^u(j)=E^u(j+1)\mbox{ and } B^E(j)E^s(j)\subseteq E^s(j+1),
	   $$
	   as desired.
	\end{proof}
	
	\begin{remark}\label{r.particular_su}
		The proof of Lemma~\ref{l.Binvariance} actually gives more information for $E^{s(u)}$ when $a_{j_0}=0$ for some $j_0\in\Z$. More precisely, cases 2.s and 3.s  imply
		\beq\label{eq:particular_s1}
		B^E(j)E^s(j)=\{\vec 0\} \mbox{ for } j=j_0, j_0+1.
		\eeq
		 Moreover, cases 2.s and 2.u imply
		 \beq\label{eq:paritcular_s2}
		 E^s(j_0+1)=\mathrm{span}\left\{\binom{0}{1}\right\}\mbox{ and } E^{u}(j_0+1)=\mathrm{span}\left\{\binom{1}{0}\right\}.
		 \eeq
		
		\end{remark}

	\begin{lemma}\label{l.Bdomination}
		There exists $N =N(M,\delta)\in \Z_{+}$ such that 
		$$\|B^E_{N}(j)\vec v_u(j)\|>2 \|B^E_N(j)\vec v_s(j)\|$$
		for all $j\in\Z$ and all unit vectors $\vec v_{s(u)}(j)\in E^{s(u)}(j)$. In particular, $E^s(j)\neq E^u(j)$ for all $j\in \Z$. 
	\end{lemma}
	\begin{proof}
		Assume for some $j\in\Z$ and $n\in\Z_+$ that $a_i \not =0$ for all $j-1\le i<j+n$. Then we always have $a_{j-1}a_j\neq 0$. Thus by Corollary~\ref{c.classificaiton_on_g}, it holds that $\vec s(j)=\binom{g_p(j)}{g_p(j-1)}$ and $g_p(i)=\phi^\a_p(i)$ for all $j-1\le i\le j+n$. Here $p=j-1$ or $j-2$. Then
		$$
		A^E_k(j)\binom{\phi^\a_p(j)}{\phi^\a_p(j-1)}=\binom{\phi^\a_p(j+k)}{\phi^\a_p(j+k-1)}\mbox{ for all }1\le k\le n.
		$$
       It implies
		\begin{align*}
		\| B^E_n(j)\vec s(j)\|&= \left\| B^E_n(j)\binom{g_p(j)}{g_p(j-1)}\right\|\\
		&=\left(\prod_{i = j}^{j+n-1} a_{i} \right)\left \|\binom{g_{p}(j+n)}{g_{p}(j+n-1)} \right \|\\
		&=\| B^E_n(j)\vec v_s(j)\| \| \vec s(j)\|
		\end{align*}
	where we set $\vec v_s(j)=\frac{\vec s(j)}{\|\vec s(j)\|}$ is a unit vector in $E^s(j)$. 
	
	Similarly, $a_{j+n-2}a_{j+n-1}\neq 0$ implies $\vec u(j+n)=\binom{g_q(j+n)}{g_q(j+n-1)}$ and $g_{q}(i)=\phi^\b_q(i)$ for all $j-1\le i\le q$. Here $q=j+n$ or $j+n+1$. Note if $a_{j+n}=0$, then \eqref{eq:u_choose_j} implies $q=j+n$. Hence, the relation between $g_q$ and $\phi^\b_q$ is still true. Then
	$$
	A^E_{-k}(j+n)\binom{\phi^\b_q(j+n)}{\phi^\b_q(j+n-1)}=\binom{\phi^\b_q(j+n-k)}{\phi^\b_q(j+n-k-1)}\mbox{ for all }1\le k\le n.
	$$
    It implies
		\begin{align*} 
			\| B^E_n(j)\vec v_u(j)\|&=\frac{\| B^E_n(j) B^E_{-n}(j+n)\vec u(j+n)\|}{\| B^E_{-n}(j+n)\vec u(j+n)\|}\\
			&=\frac{\| \vec u(j+n)\|}{\| B^E_{-n}(j+n)\vec u(j+n)\|}\\
			&=\frac{\| \vec u(j+n)\|}{\left(\prod_{i = j+n-1}^{j} \frac{1}{a_{i}} \right)\left \|\binom{g_{q}(j)}{g_{q}(j-1)} \right \|},
		\end{align*}
		where $\vec v_u(j)=\frac{B^E_{-n}(j+n)\vec u(j+n)}{\| B^E_{-n}(j+n)\vec u(j+n)\|}$ is a unit vector in $E^u(j)$. Combine the two estimates above, for all such $j$ and $n$, we obtain
		\begin{align*}
			\frac{\| B^E_n(j)\vec v_u(j)\|}{\| B^E_n(j)\vec v_su(j)\|} &=\frac{\| \vec u(j+n)\| \| \vec s(j)\|}{\left(\prod_{i = j+n-1}^{j} \frac{1}{a_{i}} \right)\left\|\binom{g_{q}(j)}{g_{q}(j-1)} \right \| \left(\prod_{i = j}^{j+n-1} a_{i} \right)\left \|\binom{g_{p}(j+n)}{g_{p}(j+n-1)} \right \|}\\
			& \ge \frac{4\| \vec u(j+n)\| \| \vec s(j)\|}{\delta^2 e^{-2 \gamma n} (1+e^{-2 \gamma})} \\
			& \ge \frac{2e^{2 \gamma n} }{C^2\delta^2}\\
			&>2,
		\end{align*}
	    provided we choose $n\ge \frac{\log(C\delta)}{\gamma}$. We may just set $N=\lceil \frac{\log(C\delta)}{\gamma}\rceil$, which depends only on $\delta=d(E,\sigma(J_{a,b}))$ and $M$. Then for all $j\in\Z$ such that $a_i\neq 0$ for all $j-1\le i< j+N$, it holds
	    $$
	    \| B^E_N(j)\vec v_u(j)\|>2\| B^E_N(j)\vec v_s(j)\|.
	    $$
	    
	Next, we consider $j$ where $a_i=0$ for some $j-1\le i<j+N$. First, we note that $B^E_N(j)\vec v_u(j)\neq \vec 0$ since Lemma~\ref{l.Binvariance} says $B^E(i)E^u(i)=E^u(i+1)$ for all $i\in\Z$. By \eqref{eq:particular_s1}, $a_i=0$ implies that
	$$
	B^E(i)\vec s(i)=B^E(i+1)\vec s(i+1)=\vec 0.
	$$
	In particular, if $i=j-1$, then $B^E(j)\vec s(j)=0$ which implies 
	$$
	B^E_N(j)\vec s(j)=B^E_{N-1}(j+1)B^E(j)\vec s(j)=\vec 0.
	$$
	If $i\ge j$, then $B^E_{i-j}(j)\vec s(j)\in E^s(i)$ which implies
	$$
    B^E_N(j)\vec s(j)=B^E_{N-i+j-1}(i+1)B^E(i)B^E_{i-j}(j)\vec s(j)=\vec 0.
	$$
	In any case, we have $B^E_N(j)\vec v_s(j)=B^E_N(j)\vec s(j)=\vec 0$ which implies for such $j$ that
	$$
	\|B^E_N(j)\vec v_u(j)\|>0=2\|B^E_N(j)\vec v_s(j)\|,
	$$
	concluding the proof.
	\end{proof}
	
   Now, the only thing left to show is that the distance between invariant directions of $B^E(j)$ is uniformly bounded away from zero for all $j \in \Z$.

	\begin{lemma}\label{l.separationofdir}
		Let $E\in\rho(J_{a,b})$. Let $E^u(j)$ and $E^s(j)$ be two invariant directions of $B^E(j)$.
		Then, there exists $\eta=\eta(\delta,M)>0$ such that
		$$
		d(E^u(j),E^s(j))> \eta \mbox{ for all }j\in\Z.
		$$
	\end{lemma}

	\begin{proof}
	Recall $\vec s(j)=\binom{g_p(j)}{g_p(j-1)}$ and $\vec u(j)=\binom{g_q(j)}{g_q(j-1)}$ where $p=j-1$ or $j-2$ and $q=j$ or $j+1$. By \eqref{eq:sphere_dist_vform} and \eqref{eq:DistofComplexLine}, we have
	$$
	d(E^u(j),E^s(j))=\frac{|\det(\vec s(j),\vec u(j))|}{\|\vec s(j)\|\|\vec u(j)\|}.
	$$
	 Recall we have $C^{-1}<\|\vec s(j)\|<C$ and $C^{-1}< \|\vec u(j)\|<C$ for all $j\in\Z$, where $C=C(\delta,M)>0$. Hence, we may sometimes instead show
	 $$
	 |\det(\vec s(j),\vec u(j))|>\eta\mbox{ for all }j\in\Z.
	 $$
	 Fix a $j\in\Z$, we again divide the proof into three cases. 
	
	Case I. Assume $a_{j-1}a_j\not =0$ or $a_{j-2}a_{j-1}\neq 0$. We can apply Corollary~\ref{c.classificaiton_on_g} to both cases. Specifically, Corollary~\ref{c.classificaiton_on_g} can be directly applied to the case $a_{j-1}a_j\neq0$. If $a_{j-2}a_{j-1}\neq 0$, we just let $j-1$ plays the role of $j$ in Corollary~\ref{c.classificaiton_on_g}. Then we just define the $j_1$ and $j_2$ appearing in Corollary~\ref{c.classificaiton_on_g} with respect to $j-1$. 
	
	Now, no matter $a_{j-1}a_j\not =0$ or $a_{j-2}a_{j-1}\neq 0$, in all the cases of Corollary~\ref{c.classificaiton_on_g}, we have for $k=j-1$ or $j$ that
	$$
	g_k(n)=\begin{cases}
		\phi^\a_k(n), & \mbox{ for all }k\le n< j_2+1;\\
		\phi^\b_k(n), & \mbox{for all }j_1<n\le k.
		\end{cases}
		$$
	Recall that $\phi^\a_p$ and $\phi^\a_k$ are linearly dependent; $\phi^\b_q$ and $\phi^\b_k$ are linearly dependent. Hence, for $k=j-1$ or $j$, we may rewrite $g_k(n)$ as:
	\beq\label{eq:green_wronskian}
	g_k(n)= \begin{cases}\frac1W\phi_p^\a(n)\phi^\b_q(k), & k\le n<j_2+1,\\
		\frac1W\phi_p^\a(k)\phi^\b_q(n), & j_1<n\le k. \end{cases}
	\eeq
	It is a standard calculation that (see. e.g. \cite{teschl})
	$$
 W[\phi^\a_p,\phi^\b](m)=a_m\big[\phi_p^\a(m+1)\phi_q^\b(m)-\phi_p^\a(m)\phi_q^\b(m+1)\big]=
	a_m\det\begin{pmatrix} \phi_p^\a(m+1)& \phi_q^\b(m+1)\\ \phi_p^\a(m)&\phi_q^\b(m)\end{pmatrix}
	$$ 
	is the modified Wronkian of $\phi^\a$ and $\phi^\b$ and is independent of $m$ as long as $a_m\neq 0$. Since we deal with singular operators where we have many different cases resulting from Corollary~\ref{c.classificaiton_on_g}, we perform a sample computation for the case $a_{j-1}a_j\neq 0$ for the sake of completeness. First, we plug the formula \eqref{eq:green_wronskian} into the equation
	$$
	a_{k-1}g_k(k-1)+a_kg_k(k+1)+(b_k-E)g_k(k)=1
	$$
	and we obtain
	$$
	\frac1{W(k)}\big[a_{k-1}\phi^\a_p(k)\phi^\b_q(k-1)+a_k\phi^\a_p(k+1)\phi^\b_q(k)+(b_k-E)\phi^\a_p(k)\phi^\b_q(k)\big]=1.
	$$
	On other hand, we have
	$$
	\phi^\a_p(k)\big[a_{k-1}\phi^\b_q(k-1)+(b_k-E)\phi^\b_q(k)+a_k\phi^\b_q(k+1)\big]=0.
	$$
	Combine the two equation above, we obtain
	$$
	\frac{a_{k}}{W(k)}\big[\phi^\a_p(k+1)\phi^\b_q(k)-\phi^\a_p(k)\phi^\b_q(k+1)\big]=1
	$$
	which implies the desired formula for $W(k)$. Note calculation above still works if $k=j-1$ and $a_{k-1}=a_{j-2}=0$. Indeed, in all cases of Corollary~\ref{c.classificaiton_on_g}, $\phi^\b$ will be the one satisfies the Dirichlet boundary condition at $k=j-1$ so all the calculations still hold true when we set $a_{k-1}=0$. $W$ is independent of $k$ because
	\begin{align*}
	W(k+1)&=a_{k+1}\det\begin{pmatrix} \phi_p^\a(k+2)& \phi_q^\b(k+2)\\ \phi_p^\a(k+1)&\phi_q^\b(k+1)\end{pmatrix}\\
	&=a_{k+1}\det \left[A^E(k)\begin{pmatrix} \phi_p^\a(k+1)& \phi_q^\b(k+1)\\ \phi_p^\a(k)&\phi_q^\b(k)\end{pmatrix}\right]\\
	&=a_{k+1}\det(A^E(k))\det \begin{pmatrix} \phi_p^\a(k+1)& \phi_q^\b(k+1)\\ \phi_p^\a(k)&\phi_q^\b(k)\end{pmatrix}\\
	&=a_{k+1}\frac{a_k}{a_{k+1}}\begin{pmatrix} \phi_p^\a(k+1)& \phi_q^\b(k+1)\\ \phi_p^\a(k)&\phi_q^\b(k)\end{pmatrix}\\
	&=a_k\begin{pmatrix} \phi_p^\a(k+1)& \phi_q^\b(k+1)\\ \phi_p^\a(k)&\phi_q^\b(k)\end{pmatrix}\\
	&=W(k).
	\end{align*}
	The calculation above clearly works for $k=j-1$ as long as $a_{j-1}a_j\neq 0$.
	In particular, $W=W(j-1)=a_{j-1}\det(\vec s(j),\vec u(j))$. Thus, if we set $m=j-1$ and $k=n=j$, then we obtain
	\beq\label{eq:large_det1}
   |\phi_p^\a(j)\phi_q^\b(j)|=|\det(\vec s(j),\vec u(j))|\cdot|a_{j-1}g_{j}(j)|\le  \frac{M}{\delta}|\det(\vec s(j),\vec u(j))|,
	\eeq
where $\delta=d(E,\sigma(J_{a,b}))=\|(J_{a,b}-E)^{-1}\|^{-1}$. Similarly, if we set if we set $m=j-1$ and $k=n=j-1$, then we obtain
	\beq\label{eq:large_det2}
|\phi_p^\a(j-1)\phi_q^\b(j-1)|=|\det(\vec s(j),\vec u(j))|\cdot|a_{j-1}g_{j-1}(j-1)|\le  \frac{M}{\delta}|\det(\vec s(j),\vec u(j))|.
	\eeq
	Note we may certainly assume $C=C(M,\delta)>0$ is large so that the following argument goes through.
	
	If $|\phi_p^\a(j)\phi_q^\b(j)|\ge C^{-4}$, then \eqref{eq:large_det1} implies the desired estimate. So we assume $|\phi_p^\a(j)\phi_q^\b(j)|< C^{-4}$. If $|\phi_p^\a(j)|<C^{-2}$ and  $|\phi_q^\b(j)|<C^{-2}$, then $\|\vec s(j)\|>C^{-1}$ and $\|\vec u(j)\|>C^{-1}$. They imply $|\phi_p^\a(j-1)|>\frac{C^{-1}}{2}$ and $|\phi_q^\b(j-1)|>\frac{C^{-1}}{2}$, which together with \eqref{eq:large_det2} implies the desired estimate. If $|\phi_p^\a(j)|\ge C^{-2}$ and $|\phi_q^\b(j)|<C^{-2}$. Then $|\phi_q^\b(j-1)|>\frac{C^{-1}}2$. Hence, if $|\phi_p^\a(j-1)|\ge \frac{C^{-2}}2$, then \eqref{eq:large_det2} implies the desired estimate. If $|\phi_p^\a(j-1)|\le \frac{C^{-2}}2$. Then we have
	\begin{align*}
	|\det(\vec s(j),\vec u(j))|&=|\phi_p^\a(j)\phi_q^\b(j-1)-\phi_q^\b(j)\phi_p^\a(j-1)|\\
	&\ge |\phi_p^\a(j)\phi_q^\b(j-1)|-|\phi_q^\b(j)\phi_p^\a(j-1)|\\
	&\ge C^{-2}\frac{C^{-1}}2-C^{-2}\frac{C^{-2}}2\\
	&\ge \frac{C^{-3}}4.
	\end{align*}
	In all the possible cases, we have $|\det(\vec s(j),\vec u(j))|>\eta>0$.

	Case II. Assume $a_{j-1}=0$. Then by Remark~\ref{r.particular_su}, we have $E^s(j)=\mathrm{span}\left\{\binom{0}{1}\right\}$ and $E^u(j)=\mathrm{span}\left\{\binom{1}{0}\right\}$. Clearly, we then have 
	$$
	d(E^s(j),E^u(j))=2,
	$$
   concluding the proof.
   
   Case III. The only case left is when $a_{j-2}=a_j=0$ and $a_{j-1}\neq 0$. Note in this case, $J_{a,b}=J_{(-,j-2]}\oplus J_{[j-1,j])}\oplus J_{(j,+)}$. Thus, eigenvalues of $J_{[j-1,j])}$ are eigenvalues of $J_{a,b}$. Let $E_1$ and $E_2$ be the two eigenvalues of $J_{[j-1,j])}$. Then we have $\delta\le |E-E_i|\le 2M$ for $i=1,2$. Note it holds that
   $$
   (J_{[j-1,j])}-E)^{-1}=\begin{pmatrix}g_{j-1}(j-1) & g_j(j-1)\\ g_{j-1}(j) & g_j(j)\end{pmatrix}.
   $$
   Since $a_{j-2}=a_j=0$, by \eqref{eq:s_choose_j-1} and \eqref{eq:u_choose_j}, we have
   \begin{align*}
   |\det(\vec s(j), \vec  u(j))|&=\left|\det\begin{pmatrix} g_{j-1}(j) & g_j(j)\\ g_{j-1}(j-1) & g_j(j-1) \end{pmatrix}\right|\\
   &=|\det  (J_{[j-1,j])}-E)^{-1}|\\
   &=|\det  (J_{[j-1,j])}-E)|^{-1}\\
   &=|(E-E_1)(E-E_2)|^{-1}\\
   &\ge \frac1{4M^2}.
   \end{align*}
Combining all three cases above, we obtain the desired estimates.
\end{proof} 
	
	Combining everything we proved in this section, we can see that $B^{E}(j)$ admits dominated splitting as defined in Definition~\ref{d.uhsequence}.

	\section{The Case with Dynamically Defined Jacobi Operators}\label{s:dynamicalVersion}
	
	In this section, we prove Theorem~\ref{t.jacobi_johnson}.  First, we introduce the definition of  $\M(2,\C)$-cocycles that have dominated splitting. Let $\Omega$ be a compact metric space $\Omega$,  $T$ be a homeomorphism $\Omega$, and $B\in C(\Omega,\M(2,\C))$ be a continuous cocycle map.  Then we consider the following dynamical system:
$$
(T,B):\Omega\times \C^2\to \Omega\times \C^2,\ (T,B)(\omega,\vec v)=(T\omega, B(\omega)\vec v).
$$
Iterations of dynamics are denoted by $(T^n,B_n(\omega)):=(T,B)^n$. In particular, similar to the sequence case, we have
\beq\label{eq:cocycle_iteration_dynamical}
B_n(\omega)=\begin{cases}B(T^{n-1}\omega)\cdots B(\omega), & n\ge1,\\ I_2 , & n=0,
\end{cases}
\eeq 
and $B_{-n}(\omega)= [B_{n}(T^{-n}\omega)]^{-1}$, $n\ge 1$, when all matrices involved are invertible. In the following definition, we again identify $z\in\C\PP^1$ with an one-dimenionsal subspace of $\C^2$ spanned by $\binom{1}{z}$ and $\infty$ with the one spanned by $\binom{0}{1}$.

\begin{defi}\label{d:domination_dynamical}
	Let $(\Omega,T)$ and $B$ be as above. Then we say $(T,B)$ has dominated spliting if there are two maps $E^s, E^u:\Omega\to \C\PP^1$ with the following properties:
	\begin{enumerate}
		\item $E^s, E^u\in C(\Omega, \C\PP^1)$. In other words, they are continuous.
		\item $B(\omega)[E^s(\omega)]\subseteq E^s(T\omega)$ and $B(\omega)[E^u(\omega)]\subseteq E^u(T\omega)$ for all $\omega\in\Omega$. 
		\item There is a $N\in\Z_+$ and $\l>1$ such that  
		$$
		\|B_N(\omega)\vec u\|> \l \|B_N(\omega)\vec s\|
		$$
		for all $\omega\in\Omega$ and all unit vectors $\vec u\in E^u(\omega)$ and $\vec s\in E^s(\omega)$.
		\end{enumerate}
	\end{defi}
\begin{remark}\label{r:dyna_DS_implies}
	Condition (3) above clearly implies that $B(\omega)E^u(\omega)\neq \{\vec 0\}$ for all $\omega\in\Omega$, which together with condition (2) implies
\beq\label{eq:nonzero_image_u}
B(\omega)E^u(\omega)=E^u(T\omega)\mbox{ for all }\omega\in\Omega.
\eeq
Moreover, condition (3) implies that $B_n(\omega)$ is nonzero for all $\omega\in\Omega$ and all $n\in\Z_+$. In particular, by compactness of $\Omega$ and continuity of $B$, we have for all $n\in\Z_+$:
\beq\label{eq:unif_lower_bound_norm}
\inf_{\omega\in\Omega}\|B_n(\omega)\|>0.
\eeq
In particular, it is quite clear (see Lemma~\ref{l:equiv_conditions_DS_dyna} below) that if $(T,B)\in\CD\CS$, then for each $\omega\in\Omega$ the sequence $B(T^{(\cdot)}\omega):\Z\to\M(2,\C)$ admits dominated splitting as defined in Definition~\ref{d.uhsequence}. The invariant directions are given by $E^s(T^{(\cdot )}\omega), E^u(T^{(\cdot)}\omega):\Omega\to\C\PP^1$. In particular, this proves Corollary~\ref{c:jacobi_johnson}.
\end{remark}
Recall we let $(T,B)\in\CD\CS$ denotes that $(T,B)$ has dominated splitting. Then, we have the following lemma.
\begin{lemma}\label{l:equiv_conditions_DS_dyna}
	To define $(T,B)\in\CD\CS$, one can replace condition (1) in Definition~\ref{d:domination_dynamical} by the following condition:
\beq\label{eq:unif_separation_us_dyn}
\inf_{\omega\in\Omega}d(E^s(\omega),E^u(\omega))>0.
\eeq
	\end{lemma}
\begin{proof}
	If $E^s$ and $E^u$ are as in Definition~\ref{d:domination_dynamical}, then condtion (2) implies that $E^s(\omega)\neq E^u(\omega)$ for all $\omega\in\Omega$. By compactness of $\Omega$ and condition (1), we then obtain \eqref{eq:unif_separation_us_dyn}. Now, suppose we have $E^s$ and  $E^u$ satisfy \eqref{eq:unif_separation_us_dyn} and conditions (2)-(3) of Definition~\ref{d:domination_dynamical}, we want to show that they are continuous on $\Omega$. Let $N$ be as in condition (3).
	
It suffices to show for every $\omega_0\in\Omega$ and every convergent sequences $\{E^s(\omega_k)\}$ and $\{E^s(\omega_k)\}$ where $\omega_k\to\omega_0$ as $k\to \infty$, we have 
	$$
		\lim_{k\to\infty} E^s(\omega_k)=E^s(\omega_0) \mbox{ and } \lim_{k\to\infty} E^u(\omega_k)=E^u(\omega_0).
	$$
 To this end, we define 
 $$
 F^u(\omega_0):=\lim_{k\to\infty} E^u(\omega_k)\mbox{ and }F^s(\omega_0):=\lim_{k\to\infty} E^s(\omega_k).
 $$ 
 First, we claim for any $n\in\Z_+$, $F^u(\omega_0)$ cannot be the eigenspace of $B_{n}(\omega_0)$ for the eigenvalue $0$, if such an eigenvalue exists for $B_n(\omega_0)$. Indeed, if $B_n(\omega_0)F^u(\omega_0)=\{\vec 0\}$, then we have $\lim_{k\to\infty}\|B_n(\omega_k)\vec u(\omega_k)\|=0$ where $\vec u(\omega)$ denotes a unit vector in $E^u(\omega)$. Thus there is a $mN>n$ such that 
 $$
 \lim_{k\to\infty}\|B_{mN}(\omega_k)\vec u(\omega_k)\|=0.
 $$
 Let $\vec s(\omega)$ d enotes a unit vector in $E^s(\omega)$. Then the equation above together with condition (3) of Definition~\ref{d:domination_dynamical} implies 
  $$
 \lim_{k\to\infty}\|B_{mN}(\omega_k)\vec s(\omega_k)\|=0.
 $$
 Since $\inf_{\Omega}d(E^s(\omega), E^u(\omega))>0$, the two estimes above implies 
 $$
 \lim_{k\to\infty}\|B_{mN}(\omega_k)\|=0
 $$
 which contradicts \eqref{eq:unif_lower_bound_norm}. This proves the claim. By continuity of $B$, $E^u(\omega_k)$ cannot be the eigenspace of $B_{n}(\omega_0)$ for the possible eigenvalue $0$ for all $k$ large. Thus for all $n\ge 1$, the following estimates hold true for all $k$ large:
 \begin{align}\label{eq:Fu_to_n}
 \nonumber B_n(\omega_0)\cdot F^u(\omega_0)&=\lim_{k\to\infty}B_n(\omega_0) \cdot E^u(\omega_k)\\
 &=\lim_{k\to\infty}B_n(\omega_k) \cdot E^u(\omega_k)\\
 \nonumber &=\lim_{k\to\infty}E^u(T^n\omega_k).
 \end{align}
 
We define a $N_0\in\N$ as follows: $N_0=0$ if $\det(B(\omega_0))=0$; otherwise,
 $$
 N_0:=\min\{n\ge 1: \det(B(T^n\omega_0))=0\}.
 $$
 Note $N_0$ may be $\infty$ and $\det[B_n(\omega_0)]\neq 0$ for all $0\le n\le N_0$. Then similarly to \eqref{eq:Fu_to_n}, we have for all $0\le n\le N_0$:
 \beq\label{eq:Fs_to_n}
 B_n(\omega_0)F^s(\omega_0)=\lim_{k\to\infty}E^s(T^n\omega_k).
 \eeq
  Then \eqref{eq:Fu_to_n}, \eqref{eq:Fs_to_n}, conditions (2)-(3) of Definition~\ref{d:domination_dynamical}, and continuity of $B$ imply for all $0\le n\le N_0$ and all $m\in\Z_+$:
  \beq\label{eq:F_domination}
  \|B_{mN}(T^n\omega_0)\vec v^u_n\|\ge \l^m \|B_{mN}(T^n\omega_0)\vec v^s_n\|
  \eeq
  where $\vec v^{s(u)}_n$ are unit vectors in $B_n(\omega_0)F^{s(u)}(\omega_0)$, respectively.  Since $E^s$ and $E^u$ satisfy \eqref{eq:unif_separation_us_dyn}, we have for all $0\le n\le N_0$ that 
  \beq \label{eq:F_unif_separation}
  d(B_n(\omega_0)F^u(\omega_0),B_n(\omega_0)F^s(\omega_0))>\delta.
  \eeq
  
  Now we are ready to show $F^s(\omega_0)=E^s(\omega_0)$. We divide it into two different cases.
  
  Case I: $N_0<\infty$. Then $\det[B(T^{N_0}\omega_0)]=0$ and $E^s(T^{N_0}\omega_0)=\ker[B(T^{N_0}\omega_0)]$. Take $n=N_0$ in \eqref{eq:F_domination}, we must have $\vec v^s_{N_0}\in E^s(T^{N_0}\omega_0)$ which implies $B_{N_0}(\omega_0)F^s(\omega_0)=E^s(T^{N_0}\omega_0)$. By $B$-invariance of $E^s$ and \eqref{eq:Fs_to_n}, we then have for all $0\le n\le N_0$:
  $$
  B_n(\omega_0)F^s(\omega_0)=E^s(T^n\omega_0).
  $$
  In particular, $F^s(\omega_0)=E^s(\omega_0)$.
  
  Case II: $N_0=\infty$. First, by Remark~\ref{r:dyna_DS_implies}, the sequence $B(T^{(\cdot )}\omega_0):\Z\to\M(2,\C)$ satisfies Defintion~\ref{d.uhsequence} with the invariant directions $E^s(T^n\omega_0)$ and $E^u(T^n\omega_0)$. On the other hand, by \eqref{eq:unif_lower_bound_norm}, \eqref{eq:F_domination}, and \eqref{eq:F_unif_separation}, $\{B_n(\omega_0)F^s(\omega_0)\}$ and $\{B_n(\omega_0)F^u(\omega_0)\}$ are also invariant directions  of $\{B(T^n\omega_0)\}_{n\ge 0}$ as described in Definition~\ref{d.uhsequence}. Hence, by Remark~\ref{r:equivalence_inv_directions}, we must have 
  $$
  E^s(T^n\omega_0)=B_n(\omega_0)F^s(\omega_0)$$ 
  for all $n\ge 0$. In particular, we have $F^s(\omega_0)=E^s(\omega_0)$.
  
  To show $F^u(\omega_0)=E^u(\omega_0)$. We define 
  $$
 N_1:=\max\{n<0: \det(B(T^n\omega_0))=0\}.
  $$
  which might be $-\infty$. Note that $\det[B(T^{n}\omega_0)]\neq 0$ for all $N_1<n<0$ if such $n$ exists. In particular, \eqref{eq:Fu_to_n}-\eqref{eq:F_unif_separation} hold true for all $N_1<n\le 0$. If $N_1>-\infty$, then $\det[B(T^{N_1}\omega_0)]=0$. By passing to a subsequence of $\{\omega_k\}$ if necessary, we may assume $E^u(T^{N_1}\omega_k)$ is convergent. By the argument showing $F^u(\omega_0)=\lim E^u(\omega_k)$ cannot the eigenspace of $B_{n}(\omega_0)$ with the possible eigenvalue $0$ for all $n\ge 1$ above, $\lim E^u(T^{N_1}\omega_k)$ cannot be the eigenspace of $B_{m}(T^{N_1}\omega_0)$ for the possible eigenvalue $0$, for all $m\ge 1$. Since $\det[B_{N_1}(T^{N_1}\omega_0)]=0$, we obtain
  \begin{align*}
F^u(\omega_0)&=\lim_{k\to\infty} E^u(\omega_k) \\
&=\lim_{k\to\infty} B_{-N_1}(T^{N_1}\omega_k)E^u(T^{N_1}\omega_k) \\
  &=B_{-N_1}(T^{N_1}\omega_0)\lim_{k\to\infty}E^u(T^{N_1}\omega_k)\\
  &=B_{-N_1}(T^{N_1}\omega_0)(\C^2)\\
  &= E^u(\omega_0).
  \end{align*}
If $N_1=-\infty$. Then similar to case 2 above, we have that $\{B_n(\omega_0)F^s(\omega_0)\}$ and $\{B_n(\omega_0)F^u(\omega_0)\}$, $n<0$, are invariant directions  of $\{B(T^n\omega_0)\}_{n<0}$ as described in Definition~\ref{d.uhsequence}. Hence, by Remark~\ref{r:equivalence_inv_directions}, we obtain $F^u(\omega_0)=E^u(\omega_0)$. This concludes the proof.
\end{proof}

The follow corollary is known. But it is hard to find an explicit proof. On the other hand, it is a relatively straightforward consequence of Lemma~\ref{l:equiv_conditions_DS_dyna} and Remark~\ref{r:dependence of constants.}, which are conseqeunces of the proofs of Theorems~\ref{l.opennessDS} and ~\ref{l:invConeFieldtoDS}. 
\begin{corollary}\label{c:stability_DS_dyna} 
	Assume $(T,B)\in\CD\CS$. Then there exists a $\e>0$ such that $(T,\tilde B)\in\CD\CS$, provided $\widetilde B\in C(\Omega,\M(2,\C))$and $\|B-\tilde B\|_\infty<\e$.
	\end{corollary}
\begin{proof}
	By Remark~\ref{r:dyna_DS_implies}, $(T,B)\in\CD\CS$ implies $B^\omega(\cdot )=B(T^{(\cdot)}\omega):\Z\to\M(2,\C)\in\CD\CS$ for all $\omega$. Moreover, the corresponding constants $\delta(B^\omega)$, $N(B^\omega)$, and $m_{B^\omega}$ are all independent of $\omega\in\Omega$. Thus by Theorem~\ref{l.opennessDS}, we can find a $\e>0$, the choice of which is independent of $\omega$, such that if $\|B-\widetilde B\|_\infty<\e$, then 
	$$
	\widetilde B^\omega(\cdot)=\widetilde B(T^{(\cdot)}\omega):\Z\to\M(2,\C)\in\CD\CS\mbox{ for all }\omega\in\Omega.
	$$ 
	Moreover, the constants  $\delta(\widetilde B^\omega)$ and $N(\widetilde B^\omega)$ are independ of $\omega\in\Omega$. Let $E^{s}_\omega$ and $E^u_\omega$ be the two invariant directions of $\widetilde B^\omega$. Thus, we can define the two invariant directions $E^s$ and $E^u$ of $(T,B)$ as 
	$$
	E^s(\omega)=E^s_\omega(0) \mbox{ and }E^u(\omega)=E^u_\omega(0)\mbox{ for all }\omega\in\Omega.
	$$
	By the discussion above, one readily checks such defined $E^s$ and $E^u$ satisfy \eqref{eq:unif_separation_us_dyn} and conditions (2) and (3). Hence, Lemma~\ref{l:equiv_conditions_DS_dyna} implies $(T,\widetilde B)\in\CD\CS$.
	\end{proof}

 Now we consider Jacobi operators which are defined dynamically. Fix functions $a\in C(\Omega,\C)$ and $b\in C(\Omega,\R)$ and define a family of Jacobi operators as in \eqref{eq:dynamical_jacobi}, i.e. 
	$$
		(J_\omega\psi)(n)=\overline{a(T^{n-1}\omega)}\psi(n-1)+a(T^n\omega)\psi(n+1)+b(T^n\omega)\psi(n).
	$$
	The Jacobi cocycle map $B^E:\Omega\to\M(2,\C)$ is given by
$$
	B^{E}(\omega)=\begin{pmatrix}E-b(\omega) &-\overline{a(T^{-1}\omega)}\\ a(\omega) &0\end{pmatrix}.
$$

	First, we have the following proposition which is basically a weak version of \cite[Theorem 6]{zhang2} since we are under the context of Jacobi operators. Since they are not entirely the same, we include the proof for sake of completeness. We set $B_r(S)=\{x:d(x,y)<r\mbox{ for some }y\in S\}$ where $S\subset \Omega$ or $\C$.
	\begin{prop}\label{p:dense_orbit_spectrum_J}
		Let $\Omega,T, a$ and $b$ be as above. Then for each $\omega\in\Omega$, each $E\in\sigma(J_\omega)$, and each $\varepsilon>0$, there exists a $\delta=\delta(\omega,\e,E)>0$ so that the following holds true:
		$$
		\mathrm{Orb}(\omega_0)\cap B_\delta(\omega)\neq\varnothing\mbox{ implies }E\subset B_\varepsilon[\sigma(H_{\omega_0})].
		$$
	In particular, if $\overline{\mathrm{Orb}(\omega_0)}=\Omega$, then $\sigma(J_\omega)\subset\sigma(J_{\omega_0})$ for all $\omega\in\Omega$. Note in particular, if $T$ is minimal, then $\sigma(J_\omega)$ is indepedent of $\omega\in\Omega$.
		\end{prop}
	\begin{proof}
		 By Weyl's criterion, it is straightforward to see for the given $\omega$, $\varepsilon$, and $E\in\sigma(J_\omega)$, we have $\|(J_\omega-E)\psi\|<\varepsilon$ for some finitely supported unit vector $\psi\in\ell^2(\Z)$. Define $\psi^{(n)}(j):=\psi(j+n)$, we see for each $n\in\Z$ it holds that
		 $$
		 \|(J_{T^n\omega}-E)\psi^{(n)}\|<\varepsilon.
		 $$
	
	By uniform continuity of $a$ and $b$ on $\Omega$ and the fact that the length of the support of $\phi^{(n)}$ is independent of $n$, there exists a $\delta=\delta(\omega,E,\e)>0$ such that the following holds true: if $d(\omega',T^n\omega)<\delta$, then
		$
		\|(J_{\omega'}-E)(\psi^{(n)})\|<\e.
		$
		In particular, if $\mathrm{Orb}(\omega_0)\cap B_\delta(\omega)\neq\varnothing$, then there is some $m\in\Z$ so that $d(T^m\omega_0,T^n\omega)<\delta$ which implies
		$$
		\|(J_{T^m\omega_0}-E)(\psi^{(n)})\|<\varepsilon.
		$$
Now, either $E\in\sigma(J_{T^m\omega_0})$; or the above inequality clearly implies 
		$$
		\|(J_{T^m\omega_0}-E)^{-1}\|>1/\varepsilon,
		$$
		which implies
		$$
		d(E,\sigma(J_{T^m\omega_0}))=\|(J_{T^m\omega_0}-E)^{-1}\|^{-1}<\e.
		$$
				In any case, we obtain $E\in B_\varepsilon[\sigma(J_{T^m\omega_0})]=B_\varepsilon[\sigma(J_{\omega_0})]$, where the last equality is a consequence of the fact that $J_{T^m\omega_0}$ and $J_{\omega_0}$ are unitarily equivalent. 
		\end{proof}
	
We are ready to prove Theorem~\ref{t.jacobi_johnson}.
\begin{proof}[Proof of Theorem~\ref{t.jacobi_johnson}]
	We let $\overline{\mathrm{Orb}(\omega_0)}=\Omega$. If $(T,B^E)\in\CD\CS$, then $E\in\cap_{\omega}\rho(J_\omega)=\rho(J_{\omega_0})$ by Corollary~\ref{c:jacobi_johnson}. On the other hand, if $E\in \rho(J_{\omega_0})$, then $d(E,\sigma(J_{\omega}))\ge \delta:=d(E,\sigma(J_{\omega_0}))$ for all $\omega\in\Omega$. By Theorem~\ref{t:johnson_sequence},
	 $$
	B^E_\omega(\cdot)=B^E(T^{(\cdot)}\omega):\Z\to\M(2,\C)\in\CD\CS \mbox { for all }\omega\in\Omega.$$ 
	Let $E^s_\omega$ and $E^u_\omega$ be the invariant directions of $B^E_\omega$. Thus, we may define
	$$
	E^s(\omega)=E^s_\omega(0)\mbox{ and }E^u(\omega)=E^u_\omega(0).
	$$
	Clearly, they are two invariant directions of $(T,B^E)$. In other words, they satisfy condition (2) of Definition~\ref{d:domination_dynamical}. By Lemmas ~\ref{l.Bdomination} and ~\ref{l.separationofdir}, the constants $\delta(B^E_\omega)$ and $N(B^E_\omega)$ of $B^E_\omega$ depend on $\delta$ only. In particular, there are $\delta$ and $N$ such that the above defined $E^s$ and $E^u$ satisfy condition (3) of Definition~\ref{d:domination_dynamical} and \eqref{eq:unif_separation_us_dyn} as well. By Lemma~\ref{l:equiv_conditions_DS_dyna}, we obtain $(T,B^E)\in\CD\CS$.
	\end{proof}

\end{document}